\documentclass{article}
\usepackage[margin=1.5in]{geometry}
\usepackage[all]{xy}
\usepackage{amsmath}
\usepackage{amsmath,amscd}
\usepackage{amsthm}
\usepackage{mathrsfs}
\usepackage{amssymb}
\usepackage{tikz}
\usepackage[utf8]{inputenc}
\usepackage[yyyymmdd]{datetime}

\newtheorem{theorem}{Theorem}[section]

\newtheorem{corollary}[theorem]{Corollary}
\newtheorem{lemma}[theorem]{Lemma}
\newtheorem*{lemma*}{Lemma}
{\theoremstyle{definition}
\newtheorem{definition}[theorem]{Definition}

\newtheorem{example}[theorem]{Example}

\newtheorem*{example*}{Example}
}
\numberwithin{equation}{section}

\newcommand{\suchthat}{\;\ifnum\currentgrouptype=16 \middle\fi|\;}
\newcommand{\scirc}{\raise1pt\hbox{$\,\scriptstyle\circ\,$}}
\newcommand{\R}{\mathbb{R}}
\newcommand{\C}{\mathbb{C}}
\newcommand{\N}{\mathbb{N}}
\newcommand{\Z}{\mathbb{Z}}

\newcommand{\hol}{\mathrm{hol}}
\newcommand{\loc}{\mathrm{loc}}
\newcommand{\bdd}{\mathrm{bdd}}

\newcommand{\LL}[1]{\mathrm{L^{1,{#1}}_{(t_0,x_0)}}}
\newcommand{\AC}[1]{\mathrm{AC}^{#1}_{(t_0,x_0)}}

\newcommand{\LC}[3]{\mathrm{L}(C^{#3}({#1}); C^{#3}({#2}))}

\newcommand{\ACC}[4]{\mathrm{AC}(\mathbb{#1}; \LC{#2}{#3}{#4})}

\begin{document}

\title{\LARGE \bf
Local and global holomorphic extensions of time-varying real analytic vector fields}

\author{Saber Jafarpour
\thanks{Saber Jafarpour is graduate student in the Department of Mathematics and Statistics,
       Queen's University, Kingston, Canada.
        {\tt\small saber.jafarpour@queensu.ca}}}

\maketitle
\thispagestyle{empty}
\pagestyle{empty}


\begin{abstract}
In this paper, we consider time-varying real analytic vector fields as
curves on the space of real analytic vector fields. Using a
suitable topology on the space of real analytic vector fields, we
study and characterize different properties of time-varying real
analytic vector fields. We study holomorphic extensions of time-varying
real analytic vector fields and show that under suitable
\textit{integrability} conditions, a time-varying real analytic vector field
on a manifold
can be extended to a time-varying holomorphic vector field on a
neighbourhood of that manifold. Moreover, we develop an operator
setting, where the \textit{nonlinear} differential equation governing the flow of a
time-varying real analytic vector field can be considered as a \textit{linear}
differential equation on an infinite dimensional locally convex vector
space.  Using the holomorphic extension results, we show that the \textit{integrability}
of the time-varying vector field ensures the convergence of the sequence
of Picard iterations for this linear differential equation. This gives us
a series representation for the flow of an integrable time-varying real analytic
vector field. We also define the exponential map between integrable time-varying
real analytic vector fields and their flows. Using the
holomorphic extensions of time-varying real analytic vector fields, we
show that the exponential map is sequentially continuous.  \vspace{0.5cm}


\end{abstract}

\section{Introduction}\label{sec:1}

The early development of the notion of \textit{real analyticity} in
mathematics has a closed connection with the development of the notion
of \textit{function}. Prior to the nineteenth
century, most of the functions used in mathematical analysis were
constructed either by applying algebraic operators on elementary functions or by a
power series except possibly at some singular points \cite{GGB:1982}. Therefore, 
mathematicians had difficulty understanding functions which are not
real analytic. It is surprising to know that Lagrange and Hankel believed that the existence of all
derivatives of a function implies the convergence of its Taylor series
\cite{GGB:1982}. It was only in the late 
nineteenth century that mathematicians started to think more carefully about the natural
question of which functions can be expanded in a Taylor series around
a point.
In 1823, Cauchy came up with a function which was
$C^{\infty}$ everywhere not real analytic at $x=0$ \cite{AC:1823}, \cite{GGB:1982}. In the
modern terminology, this function can be expressed as
\begin{equation}\label{eq:0}
f(x)=
\begin{cases}
e^{\frac{-1}{x^2}}& x\ne 0,\\
0 & x=0.
\end{cases}
\end{equation}
Starting from early twentieth
century, with the advent of the more precise notion of function,
mathematicians came up with other examples of smooth but not
real-analytic functions whose singular points have completely
different natures \cite{GGB:1982}. 

Roughly speaking, a map $f$ is real
analytic on a domain $D$ if the Taylor series of $f$ around every
point $x_0\in D$
converges to $f$ in a neighbourhood of $x_0$. By definition, for the Taylor
series of $f$ on $D$ to exist, derivatives of $f$ of any order
should exist and be continuous at every point $x_0\in D$. This means that all real
analytic maps are of class $C^{\infty}$. As is shown by the function
\eqref{eq:0}, the
converse implication is not true. In fact, given an open connected set
$\Omega\subseteq \R^n$, one can construct a family of nonzero smooth
functions on $\R^n$ which are zero on the set $\Omega$. However, by the identity
theorem, every
real analytic function which is zero on the set $\Omega$ should be
zero everywhere. This shows that the
gap between real analytic functions and smooth functions is huge \cite{SGK:2002}.

Real analytic vector fields on $\R^n$ have a close connection with the
holomorphic vector fields defined on neighbourhoods of $\R^n$ in $\C^n$. It is well-known that every real analytic vector field $f$ on $\R^n$ can be extended to a
holomorphic vector field defined on an appropriate domain in $\C^n$. However, it may
not be possible to extend the real analytic vector field $f$ to a
holomorphic vector field on the whole domain $\C^n$. This observation suggests that one should consider a
real analytic vector field as a germ of a holomorphic vector field. This
perspective for real analytic vector fields motivates the definition of a natural
topology on the space of real analytic vector fields. Unfortunately, there
does not exist a single domain such that \textit{every} real analytic
vector field on $\R^n$ can be extended to a holomorphic vector field on that
domain. The following example shows this fact.
\begin{example}
For every $n\in \N$, consider the function $f_n:\R\to \R$ defined as 
\begin{equation*}
f_n(x)=\frac{1}{1+n^2x^2},\qquad\forall x\in \R.
\end{equation*}
It is easy to see that, for every $n\in\N$, the function $f_n$ is real
analytic on $\R$. We show that there does not exist a neighbourhood
$\Omega$ of $\R$ in $\C$ such that, for every $n\in \N$, the real
analytic function $f_n$ can be extended to a holomorphic function on
$\Omega$.  Suppose that such an $\Omega$ exists. Then there exists
$r>0$ such that 
\begin{equation*}
\{x\in \C\mid \|x\|\le r\}\subseteq \Omega.
\end{equation*}
Now let $N\in \N$ be such that $\frac{1}{N}<r$ and suppose that
$\overline{f}_N$ be the holomorphic extension of $f_N$ to
$\Omega$. Then, by the identity theorem, we have
\begin{equation*}
\overline{f}_N(z)=\frac{1}{1+N^2z^2},\qquad\forall z\in \Omega.
\end{equation*}
By our choice of $N$, we have $\frac{i}{N}\in \Omega$, but $\overline{f}_N$ is not
defined at $z=\frac{i}{N}$. This is a contradiction and shows that
such an $\Omega$ does not exist.
\end{example}
Thus, the space of real analytic vector fields on $\R^n$, which we denote by $\Gamma^{\omega}(\R^n)$, can be considered as
the \textit{union} of the spaces of holomorphic vector fields defined on neighbourhoods
of $\R^n$ in $\C^n$. This process of taking union can be made precise using the
mathematical notion of \textit{inductive limit}. The space of holomorphic
vector fields on an open set $\Omega\subseteq \C^n$ has been
studied in detail in the literature \cite{kriegl1997convenient},
\cite{PWM:1980}. One can show that the so-called ``compact-open''
topology on the space of
holomorphic vector fields on $\Omega$ is generated by a family of
seminorms and thus is a locally convex
topological vector space \cite{kriegl1997convenient}. Therefore, we can represent the space of
real analytic vector fields on $\R$ as an inductive limit of a family
of locally
convex spaces. The locally convex \textit{inductive limit topology} on $\Gamma^{\omega}(\R^n)$ is
defined as the finest locally convex topology which makes all the inclusions from the
spaces of 
holomorphic vector fields to the space of real analytic vector fields continuous. 

Inductive limits of locally convex spaces arise in many
fields, including partial differential equations, Fourier analysis,
distribution theory, and holomorphic calculus. Historically, locally convex
inductive limits of locally convex spaces first appeared when mathematicians tried to define a suitable topology on the space of distributions. While there is little
literature for inductive limit of arbitrary families of locally convex
spaces, the countable inductive limit of locally convex spaces is rich
in both theory and applications. The importance of the connecting maps in inductive limits of locally
convex spaces was first realized by Jos{\'e} Sebasti{\~a}o e Silva \cite{JSS:1955}. Motivated by
studying the space of germs of holomorphic functions, Sebasti{\~a}o e
Silva investigated inductive limit of locally convex spaces with compact connecting maps. Inductive
limits with weakly compact connecting maps were studied later
by Komatsu in \cite{HK:1967}, where he showed that weakly compact
inductive limits share many nice properties with the compact
inductive limits.

Unfortunately, the space of real analytic vector fields on $\R^n$
is not the inductive limit of a countable family of locally convex spaces. However, it is possible to represent
the space of germs of holomorphic vector fields around a \textit{compact
set} as the inductive limit of a countable family of locally convex
spaces with compact connecting maps \cite[Theorem 8.4]{kriegl1997convenient}. 
Let $\{K_i\}_{i\in \N}$ be a family of compact sets on $\R^n$ such that
$\bigcup_{i=1}^{\infty} K_i=\R^n$ and 
\begin{equation*}
\mathrm{cl}(K_i)\subseteq K_{i+1},\qquad\forall i\in \N.
\end{equation*}
It is interesting to note that the space of real analytic vector fields on
$\R^n$ can also be obtained by \textit{gluing together} the vector spaces
of germs of holomorphic vector fields on compact sets $\{K_i\}_{i\in I}$. The concept of \textit{gluing together}
mentioned above can be made precise using the notion of projective limit of vector
spaces. The coarsest locally convex topology on $\Gamma^{\omega}(\R)$ which makes all
the gluing maps continuous is called the \textit{projective limit topology} on
$\Gamma^{\omega}(\R^n)$. Having defined the \textit{inductive limit topology} and
\textit{projective limit topology} on the space of real analytic vector fields
on $\R^n$, it would be interesting to study the relation between these
two topologies. As to our knowledge,
the first paper that studied the relation between these two topologies
on the space of real analytic vector fields is \cite{AM:1966}, where it is
shown that these two topologies are identical. There has been a recent interest in this topology and its applications in the theory
of partial differential equations \cite{JB-PD:2001}, \cite{ML:2000}.

Time-varying vector fields and their flows arise naturally in
studying physical problems. In particular, in some branches of applied sciences
such as control theory, it is essential to work with time-varying
vector fields whose dependence on time is only
\textit{measurable}. Existence and uniqueness of flows of 
time-varying vector field has been deeply studied in the literature \cite[Chapter 2]{coddington1955theory}. However, theory of time-varying vector
fields with measurable dependence on time and their flows is not as
well-developed as theory of time-invariant vector fields. In this paper, we study time-varying real
analytic vector
fields on a manifold $M$ by considering them as curves on the vector
space $\Gamma^{\omega}(TM)$. Using the $C^{\omega}$-topology on the space of real analytic
vector fields, different properties of this curve can be studied and
characterized. In particular, we can use the framework in \cite{RB-AD:2011} to
define and characterize the \textit{Bochner integrability} of curves on $\Gamma^{\omega}(TM)$.

It is well-known that every real analytic vector fields
can be extended to a holomorphic vector field on a complex
manifold. Consider a time-varying real analytic
vector field on $M$ with some regularity in time. It is interesting to study whether this
time-varying real analytic vector field can be
extended to a time-varying holomorphic vector field on a complex
manifold containing $M$. Unfortunately this holomorphic extension is not generally
possible. As the following
example shows, a measurable
time-varying real analytic vector field may not even have a local holomorphic
extension to a complex manifold.
\begin{example}\label{ex:2}
Let $X:\R\times \R\to T\R$ be a time-varying vector field defined as 
\begin{equation*}
X(t,x)=
\begin{cases}
\frac{t^2}{t^2+x^2}\frac{\partial}{\partial x} & x\ne 0 \mbox{ or } t\ne 0,\\
0 & x,t=0.
\end{cases}
\end{equation*}
Then $X$ is a time-varying vector field on $\R$ which is
locally integrally bounded with respect to $t$ and real analytic with respect to
$x$. However, there does not exist connected neighbourhood $\overline{U}$ of $x=0$ in $\C$ on which
$X$ can be extended to a holomorphic vector field. To see this, let
$\overline{U}\subseteq \C$ be a connected neighbourhood of $x=0$ and let
$\mathbb{T}\subseteq\R$ be a neighbourhood of $t=0$. Let
$\overline{X}:\mathbb{T}\times \overline{U}\to T\C$
be a time-varying vector field which is measurable in time and holomorphic in
state such that
\begin{equation*}
\overline{X}(t,x)=X(t,x)\qquad\forall x\in \R\cap \overline{U},\ \forall
t\in \mathbb{T}.
\end{equation*}
Since $0\in \mathbb{T}$, there exists $t\in \mathbb{T}$ such that $\mathrm{cl}(D(0,t))\subseteq
\overline{U}$. Let us fix this $t$ and define the real analytic vector field $X_t:\R\to T\R$ as 
\begin{equation*}
X_t(x)=\frac{t^2}{t^2+x^2}\frac{\partial}{\partial x},\qquad\forall x\in \R,
\end{equation*}
and the holomorphic vector field $\overline{X}_t:\overline{U}\to T\C$ as
\begin{equation*}
\overline{X}_t(z)=\overline{X}(t,z)\qquad\forall z\in \overline{U},
\end{equation*}
Then it is clear that $\overline{X}_t$ is a holomorphic extension of
$X_t$.  However, one can define another holomorphic vector field $Y:D(0,t)\to T\C$ by
\begin{equation*}
Y(z)=\frac{t^2}{t^2+z^2}\frac{\partial}{\partial z},\qquad\forall z\in D(0,t),
\end{equation*}
It is easy to observe that $Y$ is also a holomorphic extension of
$X_t$. Thus, by the identity theorem, we should have
$Y(z)=\overline{X}_t(z)$, for all $z\in D(0,t)$. Moreover, we should
have $\overline{U}\subseteq D(0,t)$. However, this is a contradiction
with the fact that $\mathrm{cl}(D(0,t))\subseteq \overline{U}$. 
\end{example}

As the above example suggests, without any joint condition on time and
space, it is impossible to prove any holomorphic extension of a
time-varying real analytic vector field to a time-varying holomorphic
vector field. It turns
out that \textit{local Bochner integrability} is the right joint condition
for a time-varying real analytic vector field to ensure the existence of a holomorphic extension. Using the inductive limit characterization
of the space of real analytic vector fields, we show that the \textit{global}
extension of \textit{locally Bochner integrable} time-varying real analytic vector fields is possible.
More specifically, we show that, for a locally Bochner
integrable time-varying real analytic vector field $X$ on $M$, there
exists a locally Bochner integrable time-varying holomorphic vector
field defined on a
neighbourhood of $M$ which agrees with $X$ on $M$. We call this result a \textit{global} extension since it proves the existence of the holomorphic extension of a time-varying vector field to a
neighbourhood of its \textit{whole} state domain.

In order to study the holomorphic extension of a \textit{single}
locally Bochner integrable time-varying real analytic vector field, the global extension result
is a useful tool. However, this extension theorem is indecisive when it comes to
questions about holomorphic extension of all elements of a family of
locally Bochner integrable time-varying real analytic vector fields to a \textit{single} domain. Using
the projective limit characterization of space of real analytic
vector fields, we show that one can \textit{locally} extend every element
of a bounded family of locally Bochner integrable time-varying real
analytic vector fields to a locally Bochner integrable time-varying holomorphic vector field defined on a single domain.

The connection between time-varying vector
fields and their flows is of fundamental importance in the
theory of differential equations and mathematical control theory. The operator approach for studying time-varying vector fields and their
flows in control theory started with the work of Agrachev and Gamkrelidze
\cite{agrachev1978exponential}. One can also find traces of
this approach in the nilpotent Lie approximations for studying controllability of
systems \cite{sussmannlie1983}, \cite{sus1987}. In \cite{agrachev1978exponential} a framework is proposed for
studying complete time-varying vector fields and their flows. The cornerstone of this
approach is the space $C^{\infty}(M)$, which is both an $\R$-algebra
and a locally convex vector space. In this framework, a smooth vector
field on $M$ is considered as a derivation of $C^{\infty}(M)$ and a smooth
diffeomorphism on $M$ is considered as a unital\/ $\R$-algebra
isomorphism of $C^{\infty}(M)$. Using a family of seminorms on
$C^{\infty}(M)$, weak topologies on the space of derivations of
$C^{\infty}(M)$ and on the space of unital\/ $\R$-algebra
isomorphisms of $C^{\infty}(M)$ are defined
\cite{agrachev1978exponential}. Then a time-varying vector field is considered as a curve on the space of derivations of
$C^{\infty}(M)$ and its flow is considered as a curve on the space of $\R$-algebra
isomorphisms of $C^{\infty}(M)$. While this framework seems to be
designed for smooth vector fields and their flows, in
\cite{agrachev1978exponential} and \cite{agrachevcontrol2004} the
same framework is used for studying time-varying \textit{real analytic} vector fields and their flows. In \cite{agrachev1978exponential}, using the characterizations of
vector fields as derivations and
their flows as unital algebra isomorphism, the \textit{nonlinear} differential equation on $\R^n$ for flows of a complete time-varying vector field
is transformed into a \textit{linear} differential equation on the
infinite-dimensional locally convex space
$\LC{\R^n}{\R^n}{\infty}$. While working with linear differential
equations seems to be more desirable than working with their nonlinear
counterparts, the fact that the underlying space of this linear
differential equation is an infinite-dimensional locally convex spaces
makes this study complicated. In fact, the theory of linear ordinary differential
equations on a locally convex spaces is completely different from the
classical theory of linear differential equations on $\R^n$ or Banach
spaces \cite{SGL-GOS:1994}. In \cite{agrachev1978exponential} it has been shown that, if the
vector field is integrable in time, real analytic in state, and has a bounded holomorphic
extension to a neighbourhood of $\R^n$, the sequence of Picard iterations
for the linear infinite-dimensional differential equation converges in
$\LC{\R^n}{\R^n}{\infty}$. In this case, one
can represent flows of a time-varying real analytic system as a series
of iterated composition of the time-varying vector field. 

In this paper, in order to study real analytic vector fields
and their flows in a consistent way, we can
extend the operator approach of
\cite{agrachev1978exponential} by replacing the locally
convex space $C^{\infty}(M)$ with $C^{\omega}(M)$. In particular, using the result of
\cite{JG:1981}, we show that there is a one-to-one
correspondence between real analytic vector fields on $M$ and
derivations of $C^{\omega}(M)$. Moreover, using the results of \cite{AEM:1952}, we show that
$C^{\omega}$-maps are in one-to-one correspondence with unital\/ $\R$-algebra
homomorphisms on $C^{\omega}(M)$. Thus, using the fact that time-varying real analytic vector
fields and their flows are curves on $\LC{M}{M}{\omega}$, we translate
the \textit{nonlinear} differential equation governing the flow a time-varying
real analytic vector field into a \textit{linear} differential equation on
$\LC{M}{M}{\omega}$. In the real analytic case, we show that a solution for the \textit{linear} differential
equation of a locally integrally bounded time-varying real
analytic vector field exists and is unique. In particular, using a
family of generating seminorms on the space of real analytic functions, we show that the
sequence of Picard iterations for our \textit{linear}
differential equation on the locally convex space $\LC{M}{M}{\omega}$
converges. This will generalize the result of \cite[Proposition 2.1]{agrachev1978exponential} to the case of locally Bochner integrable time-varying real
analytic vector fields.

Finally, we define the exponential map
between locally integrally bounded time-varying real analytic vector fields and their
flows. Using the sequence of Picard iteration for flows of
time-varying vector fields, we show that the exponential map is
sequentially continuous.

\section{Mathematical Notations}\label{sec:2}

In this section, we introduce the mathematical notations that we use in
this paper.

Let $r\in \R^{>0}$ and $x_0\in \R^n$,  we denote the disk of radius $r$ with center $x_0$ by $\mathrm{D}(x_0,r)$.
A \textbf{multi-index of order $m$} is an element
$(r)=(r_1,r_2,\ldots,r_m)\in (\Z_{\ge
  0})^m$. For all multindices $(r)$ and $(s)$ of order $m$, every
$x=(x_1,x_2,\ldots,x_m)\in \R^m$, and every $f:\R^m\to\R^n$, we define
\begin{eqnarray*}
|(r)|&=&r_1+r_2+\ldots+r_m,\\
(r)+(s)&=&(r_1+s_1,r_2+s_2,\ldots,r_m+s_m),\\
(r)!&=&r_1!r_2!\ldots r_m!,\\
x^{(r)}&=&x_1^{r_1}x_2^{r_2}\ldots x_m^{r_m},\\
D^{(r)}f(x)&=&\frac{\partial^{|r|} f}{\partial
               x_1^{r_1}\partial x_2^{r_2}\ldots \partial x_m^{r_m}},\\
\binom{(r)}{(s)}&=&\binom{r_1}{s_1}\binom{r_2}{s_2}\ldots \binom{r_m}{s_m}.
\end{eqnarray*}
We denote the multi-index $(0,0,\ldots,1,\ldots,0)\in (\Z_{\ge
  0})^m$, where $1$ is in the $i$-th place, by $\widehat{(i)}$. One
can compare multindices $(r),(s)\in (\Z_{\ge 0})^m$. We say that
$(s)\le (r)$ if,
for every $i\in\{1,2,\ldots,m\}$, we have $s_i\le r_i$.

The space of all decreasing sequences $\{a_i\}_{i\in\N}$ such that
$a_i\in \R_{>0}$ and $\lim_{n\to\infty} a_n=0$ is denoted by
$\mathbf{c}^{\downarrow}_{0}(\Z_{\ge 0};\R_{>0})$.

For the space $\R^n$, we define the Euclidean norm $\|.\|_{\R^n}:\mathbb{R}^n\to \R$ as
\begin{equation*}
\|\mathbf{v}\|_{\R^n}=\left(v_1^2+v_2^2+\ldots+v_n^2\right)^{\frac{1}{2}},\qquad\forall \mathbf{v}\in\R^n.
\end{equation*}
For the space $\C^n$, we define the norm $\|.\|_{\C^n}:\C^n\to \R$ as
\begin{equation*}
\|\mathbf{v}\|_{\C^n}=\left(v_1\overline{v}_1+v_2\overline{v}_2+\ldots+v_n\overline{v}_n\right)^{\frac{1}{2}},\qquad\forall \mathbf{v}\in\C^n.
\end{equation*}

Let $M$ be an $n$-dimensional $C^{\nu}$-manifold, where $\nu\in
\{\omega,\hol\}$ and let $(U,\phi)$ be a coordinate chart on $M$. Then we define $\|.\|_{(U,\phi)}:U\to \R$ as
\begin{equation*}
\|x\|_{(U,\phi)}=\|\phi(x)\|_{\mathbb{F}^n},\qquad\forall x\in U.
\end{equation*}

Let $M$ be an $n$-dimensional $C^{\nu}$-manifold , where $\nu\in
\{\omega,\hol\}$, $(U,\phi)$ be a coordinate chart
on $M$, and $f$ be a $C^{\nu}$-function on $M$. Then, for every
multi-index $(r)$, we define $\|D^{(r)}f(x)\|_{(U,\phi)}$ as
\begin{equation*}
\|D^{(r)}f(x)\|_{(U,\phi)}=\|D^{(r)}\left(f\scirc\phi\right)(\phi^{-1}(x))\|_{\mathbb{F}},\qquad\forall
x\in U.
\end{equation*}
When the coordinate chart on $M$ is understood from the context, we usually omit the
subscript $(U,\phi)$ in the norm. 

For every $C^{\nu}$-vector field $X$ and every multi-index $(r)$, we define $\|D^{(r)}X(x)\|_{(U,\phi)}$ as
\begin{equation*}
\|D^{(r)}X(x)\|_{(U,\phi)}=\|D^{(r)}\left(T\phi\scirc X\scirc\phi^{-1}\right)(\phi(x))\|_{\mathbb{F}},\qquad\forall
x\in U.
\end{equation*}
When the coordinate chart on $M$ is understood from the context, we usually omit the
subscript $(U,\phi)$ in the norm.

In this paper, we only study holomorphic and real analytic regularity
classes. We usually denote $C^{\hol}$ for the holomorphic regularity
and $C^{\omega}$ for the real analytic regularity.  Let $M$ be a real analytic manifold, we denote the
space of real analytic functions on $M$ by $C^{\omega}(M)$ and the
space of real analytic vector fields on $M$ by
$\Gamma^{\omega}(TM)$. Similarly, for a complex manifold $M$, we
denote the space of holomorphic functions on $M$ by $C^{\hol}(M)$ and
the space of holomorphic vector fields on $M$ by $\Gamma^{\hol}(TM)$. 

We denote the Lebesgue measure on $\R$ by $\mathfrak{m}$. Let
$\mathbb{T}\subseteq \R$ be an interval. Then we denote the space of
integrable functions on $\mathbb{T}$ by  
$\mathrm{L}^1(\mathbb{T})$.
\begin{equation*}
\mathrm{L}^1(\mathbb{T})=\left\{f:\mathbb{T}\to \R\suchthat
  \int_{\mathbb{T}} |f|d\mathfrak{m}<\infty\right\}.
\end{equation*}
The space of continuous functions on $\mathbb{T}$ is denoted by
$\mathrm{C}^0(\mathbb{T})$.

Let $V$ be a locally convex space on the field $\mathbb{F}$. Then the space of all linear
continuous functionals from $V$ to $\mathbb{F}$ is the topological
dual of $V$ and is denoted by $V'$. We usually denote the space $V'$
endowed with the weak topology by $V'_{\sigma}$ and the space $V'$
endowed with the strong topology by $V'_{\beta}$.  

Let $V$ and $W$ be two locally convex spaces on the field $\mathbb{F}$. Then we
denote their tensor product by $V\otimes W$. The projective tensor
product of $V$ and $W$ is denoted by $V \otimes_{\pi} W$ and the
injective tensor product of $V$ and $W$ is denoted by
$V\otimes_{\epsilon}W$. The completion of vector spaces $V
\otimes_{\pi} W$ and $V\otimes_{\epsilon}W$ are denoted by
$V \widehat{\otimes}_{\pi}W$ and $V \widehat{\otimes}_{\epsilon}W$, respectively.

Let $\Lambda$ be a set. A binary relation $\succeq$ 
\textbf{directs} $\Lambda$ if
\begin{enumerate}
\item for every $i,j,k\in \Lambda$, $i\succeq j$ and $j\succeq k$
  implies $i\succeq k$,
\item for every $i\in \Lambda$, we have $i\succeq i$,
\item for every $i,j\in \Lambda$, there exists $m\in \Lambda$ such
  that $m\succeq i$ and $m\succeq j$.
\end{enumerate}
A \textbf{directed set} is a pair $(\Lambda,\succeq)$ such that
$\succeq$ directs $\Lambda$.

Let $\Lambda$ be a directed set and $\{V_{\alpha}\}_{\alpha\in
  \Lambda}$ be a family of objects indexed by the elements in the set
$\Lambda$ and, for every $\alpha,\beta\in \Lambda$ such that
$\alpha\succeq\beta$, there exists a morphism $f_{\alpha,\beta}:V_{\alpha}\to
V_{\beta}$ such that 
\begin{enumerate}
\item $f_{\alpha,\alpha}=\mathrm{id}$, for every $\alpha\in \Lambda$, and
\item $f_{\alpha,\gamma}=f_{\beta,\gamma}\scirc f_{\alpha,\beta}$, for
  every $\alpha\succeq\beta\succeq\gamma$.
\end{enumerate}
Then, the pair $(V_{\alpha},\{f_{\alpha,\beta}\})$ is called an
inductive family of objects. 

Let $(V_{\alpha},\{f_{\alpha,\beta}\})$ be an
inductive family of
objects. Then we denote the inductive limit of
$(V_{\alpha},\{f_{\alpha,\beta}\})$ by 
\begin{equation*}
\varinjlim V_{\alpha}
\end{equation*}

Let $\Lambda$ be a directed set and $\{V_{\alpha}\}_{\alpha\in
  \Lambda}$ be a family of objects indexed by the elements in the set
$\Lambda$ and, for every $\alpha,\beta\in \Lambda$ such that
$\alpha\succeq\beta$, there exists a morphism $f_{\alpha,\beta}:V_{\beta}\to
V_{\alpha}$ such that 
\begin{enumerate}
\item $f_{\alpha,\alpha}=\mathrm{id}$, for every $\alpha\in \Lambda$, and
\item $f_{\alpha,\gamma}=f_{\alpha,\beta} \scirc f_{\beta,\gamma}$, for
  every $\alpha\succeq\beta\succeq\gamma$.
\end{enumerate}
Then, the pair $(V_{\alpha},\{f_{\alpha,\beta}\})$ is called a
projective family of objects. 

Let $(V_{\alpha},\{f_{\alpha,\beta}\})$ be a projective family of
objects. Then we denote the projective limit of
$(V_{\alpha},\{f_{\alpha,\beta}\})$ by 
\begin{equation*}
\varprojlim V_{\alpha}
\end{equation*}

\section{Holomorphic extension of real analytic mappings}\label{sec:3}

In this section, we review some of the well-known results about
extension of ``time-invariant'' real analytic functions and vector
fields. Since every real analytic mapping is defined on a real
analytic manifold, the first step for studying holomorphic extensions of such
mappings is to extend the underlying real analytic manifold to a
complex manifold. We start with definition of totally real
submanifolds of complex manifolds.

\begin{definition}
Let $M$ be a complex manifold with an almost complex structure $J$. A
submanifold $N$ of $M$ is called \textbf{a totally real submanifold} if, for every $p\in N$,
we have $J(T_{p}N)\bigcap T_pN=\{0\}$. 
\end{definition}

It can be shown that, for every real analytic manifold $M$, there exists a
complex manifold $M^{\C}$ which contains $M$ as a totally real
submanifold \cite{HW-FB:1959}.
\begin{theorem}
Let $M$ be a real analytic manifold. There exists a complex
manifold $M^{\C}$ such that
$\mathrm{dim}_{\C}M^{\C}=\mathrm{dim}_{\R}M$ and $M$ is a totally
real submanifold of $M^{\C}$.
\end{theorem}
The complex manifold $M^{\C}$ is called a \textbf{complexification}
of the real analytic manifold $M$.

Now that we can extended the real analytic manifolds to a complex manifold, it
is time to study holomorphic extensions of real analytic mappings on
the complexification of their domains. One can show that every real analytic function (vector field) on
$M$ can be extended to a holomorphic function (vector field) on some complexification
of $M$.
\begin{theorem}
Let $M$ be a real analytic manifold and $X:M\to TM$ be a real analytic
vector field on $M$. Then there
exists a complexification of $M$ denoted by $M^{\C}$ and a
holomorphic vector field $\overline{X}:M^{\C}\to TM^{\C}$ such
that 
\begin{equation*}
X(x)=\overline{X}(x),\qquad\forall x\in M.
\end{equation*}
\end{theorem}
The vector field $\overline{X}$ is called a \textbf{holomorphic extension} of the vector field $X$.

\section{Real analytic vector fields as derivations on $C^{\omega}(M)$}\label{sec:4}

In this section, we characterize real analytic vector fields as
derivations on the $\R$-algebra $C^{\omega}(M)$. We will see that this
characterization plays an important role in studying flows of
time-varying vector fields.

Let $M$ be a real analytic manifold and let $X:M\to TM$ be a real
analytic vector field
on $M$. Then we define the corresponding linear map $\widehat{X}:C^{\omega}(M)\to C^{\omega}(M)$ as
\begin{equation*}
\widehat{X}(f)=df(X),\qquad\forall f\in C^{\omega}(M).
\end{equation*}
Using the Leibniz rule, this linear map can be shown to be a derivation on the $\mathbb{R}$-algebra $C^{\omega}(M)$.

More interestingly, one can show there is a one-to one
correspondence between $C^{\omega}$-vector fields on $M$ and derivations
on the $\mathbb{R}$-algebra $C^{\omega}(M)$. 

\begin{theorem}\label{th:9}
Let $M$ be a real analytic manifold. If $X$ is a real analytic vector field, then $\widehat{X}$ is a derivation
on the $\mathbb{R}$-algebra $C^{\omega}(M)$. Moreover, for every
derivation $D:C^{\omega}(M)\to C^{\omega}(M)$, there exists a $C^{\omega}$-vector field $X$ such that $\widehat{X}=D$.
\end{theorem}
\begin{proof}
The sketch of proof is given in \cite[Theorem 4.1]{JG:1981}
\end{proof}

\section{Real analytic maps as unital $\R$-algebra homomorphism on $C^{\omega}(M)$}\label{sec:4'}

In this section, we characterize real analytic mappings as unital
$\R$-algebra homomorphisms on $C^{\omega}(M)$.

Let $\phi:M\to N$ be a real analytic map. Then we can define
the associated map $\widehat{\phi}:C^{\omega}(N)\to C^{\omega}(M)$ as 
\begin{equation*}
\widehat{\phi}(f)=f\scirc\phi.
\end{equation*} 
It is easy to see that $\widehat{\phi}$ is an $\mathbb{R}$-algebra
homomorphism. For every $x\in M$, one can define the unital\/ $\mathbb{R}$-algebra homomorphism
$\mathrm{ev}_x:C^{\omega}(M)\to\mathbb{R}$ as
\begin{equation*}
\mathrm{ev}_x(f)=f(x).
\end{equation*}
The map $\mathrm{ev}_x$ is called the evaluation map at $x\in
M$. The evaluation map plays an essential role in
characterizing unital\/ $\mathbb{F}$-algebra homomorphisms. The following result is of significant importance. 

\begin{theorem}\label{th:10}
Let $M$ be a real analytic manifold. Let $\phi:C^{\omega}(M)\to \mathbb{R}$ be a nonzero and unital\/ $\mathbb{R}$-algebra homomorphism. Then
there exists $x\in M$ such that $\phi=\mathrm{ev}_{x}$.
\end{theorem}
\begin{proof}

For the case when $M$ and $N$ are open subsets of an Euclidean space,
the proof for this theorem is given in \cite[Theorem 2.1]{PD-ML:2003}. However, it seems that
this proof cannot be generalized to include the general real analytic
manifolds. Using the techniques and ideas in \cite[Proposition 12.5]{AEM:1952}, we
present a proof of this theorem for the general case. Let
$\phi:C^{\omega}(M)\to \R$ be a unital\/ $\R$-algebra homomorphism. It is easy
to see that $\mathrm{Ker}(\phi)$ is a maximal ideal in $C^{\omega}(M)$. For
every $f\in C^{\omega}(M)$, we define
\begin{equation*}
Z(f)=\{x\in M\mid f(x)=0\}.
\end{equation*} 
\begin{lemma*}\label{lem:3}
Let $n\in \N$ and $f_1,f_2,\ldots,f_n\in \mathrm{Ker}(\phi)$. Then we
have 
\begin{equation*}
\bigcap_{i=1}^{n} Z(f_i)\ne\emptyset.
\end{equation*}
\end{lemma*}
\begin{proof}
Suppose that we have 
\begin{equation*}
\bigcap_{i=1}^{n} Z(f_i)=\emptyset.
\end{equation*}
Then we can define a function $g\in C^{\omega}(M)$ as
\begin{equation*}
g(x)=\frac{1}{(\sum_{i=1}^{n} (f_i(x))^2)},\qquad\forall x\in M.
\end{equation*}
Then it is clear that we have
\begin{equation*}
\left(\sum_{i=1}^{n} (f_i)^2\right)(g)=\mathbf{1},
\end{equation*}
where $\mathbf{1}:C^{\nu}(M)\to \mathbb{F}$ is a unital\/
$\mathbb{F}$-algebra homomorphism defined as
\begin{equation*}
\mathbf{1}(f)=1.
\end{equation*} 
Since $\mathrm{Ker}(\phi)$ is an ideal in $C^{\omega}(M)$, we have
$\mathbf{1}\in \mathrm{Ker}(\phi)$. This implies that $\phi=0$\@, which
is a contradiction of $\phi$ being unital.
\end{proof}

Since $M$ is a real analytic manifold, there exists a $C^{\omega}$-embedding of $M$ into some $\R^N$ (one can use Grauert's embedding theorem with $N=4n+2$). Let $x_1,x_2,\ldots,x_N$ be the standard
coordinate functions on $\R^N$ and
$\widehat{x}_1,\widehat{x}_2,\ldots,\widehat{x}_N$ be their
restrictions to $M$. Now, for every $i\in \{1,2,\ldots,N\}$, consider the functions
$\widehat{x}_i-\phi(\widehat{x}_i)\mathbf{1}\in C^{\omega}(M)$. It is
easy to see that
\begin{equation*}
\phi(\widehat{x}_i-\phi(\widehat{x}_i)\mathbf{1})=\phi(\widehat{x}_i)-\phi(\widehat{x}_i)\phi(\mathbf{1})=0,\qquad\forall i\in \{1,2,\ldots,N\}.
\end{equation*}
This implies that, for every $i\in \{1,2,\ldots,N\}$, we have
$\widehat{x}_i-\phi(\widehat{x}_i)\mathbf{1}\in
\mathrm{Ker}(\phi)$. So, by the above
Lemma, we get
\begin{equation*}
\bigcap_{i=1}^{N} Z(\widehat{x}_i-\phi(\widehat{x}_i)\mathbf{1})\ne \emptyset. 
\end{equation*}
Since $x_1,x_2,\ldots,x_N$ are coordinate functions, it is easy to see that
$\bigcap_{i=1}^{N} Z(\widehat{x}_i-\phi(\widehat{x}_i)\mathbf{1})$ is
just a one-point set. So we set $\bigcap_{i=1}^{N}
Z(\widehat{x}_i-\phi(\widehat{x}_i)\mathbf{1})=\{x\}$. 

Now we proceed to prove the theorem. 
Note that, for every $f\in \mathrm{Ker}(\phi)$,
we have
\begin{equation*}
Z(f)\cap\{x\}=Z(f)\bigcap \left(\cap_{i=1}^{N}
Z(\widehat{x}_i-\phi(\widehat{x}_i)\mathbf{1})\right).
\end{equation*}
So, by the above Lemma, we have 
\begin{equation*}
Z(f)\cap\{x\}\ne \emptyset,\qquad\forall f\in \mathrm{Ker}(\phi).
\end{equation*}
This implies that 
\begin{equation*}
\{x\}\subseteq Z(f),\qquad\forall f\in \mathrm{Ker}(\phi).
\end{equation*}
This means that 
\begin{equation*}
\{x\}\subseteq \bigcap_{f\in \mathrm{Ker}(\phi)}Z(f).
\end{equation*}
This implies that $\mathrm{Ker}(\phi)\subseteq \mathrm{Ker}(\mathrm{ev}_x)$. Since $\mathrm{Ker}(\mathrm{ev}_x)$ and
$\mathrm{Ker}(\phi)$ are both maximal ideals, we have
\begin{equation*}
\mathrm{Ker}(\mathrm{ev}_x)=\mathrm{Ker}(\phi).
\end{equation*}
Now let $f\in C^{\omega}(M)$, so we have $f-f(x)\mathbf{1}\in
\mathrm{Ker}(\phi)$. This implies that
\begin{equation*}
0=\phi(f-f(x)\mathbf{1})=\phi(f)-f(x).
\end{equation*}
So, for every $f\in C^{\omega}(M)$,
\begin{equation*}
\phi(f)=f(x).
\end{equation*}
Therefore, we have $\phi=\mathrm{ev}_x$.
\end{proof}

\begin{theorem}\label{th:1}
Let $M$ and $N$ be real analytic manifolds. Then, for every $\mathbb{R}$-algebra map $A:C^{\omega}(M)\to C^{\omega}(N)$, there exists
a real analytic map $\phi:N\to M$ such that 
\begin{equation*}
\widehat{\phi}=A.
\end{equation*}
\end{theorem}
\begin{proof}
For every $x\in N$, consider the unital\/ $\mathbb{R}$-algebra homomorphism
$\mathrm{ev}_{x}\scirc A:C^{\omega}(M)\to \mathbb{R}$. By Theorem \ref{th:10}, there
exists $y_x\in M$ such that $\mathrm{ev}_x\scirc A=\mathrm{ev}_{y_x}$. We
define $\phi:N\to M$ as
\begin{equation*}
\phi(x)=y_x,\qquad\forall x\in N.
\end{equation*}

Let $(U,\eta=(x^1,x^2,\ldots,x^m))$ be a coordinate neighbourhood on $M$
around $y_x$. Then, by using the Grauert's embedding theorems, there exist functions
$\tilde{x}^1,\tilde{x}^2,\ldots,\tilde{x}^m$ such that, for every
$i\in \{1,2\ldots,m\}$, we have
\begin{eqnarray*}
\tilde{x}^i&\in& C^{\omega}(N),\\
\tilde{x}^i\lvert_{U}&=&x^i.
\end{eqnarray*}
Thus, for every $x\in U$, we have
\begin{equation*}
y^i_x=\mathrm{ev}_x\scirc
A(\tilde{x}^i)=A(\tilde{x}^i)(x),\qquad\forall i\in\{1,2,\ldots,m\}.
\end{equation*}
However, for every $i\in\{1,2,\ldots,m\}$, we have $A(\tilde{x}^i)\in
C^{\omega}(N)$. This implies that, for every $i\in \{1,2,\ldots,m\}$, the
function $y^i_x$ is real analytic with respect to $x$ on the
neighbourhood $U$.  Therefore, the map $\phi$ is real analytic. One can easily check that $\widehat{\phi}=A$.
\end{proof}

\section{Inductive limit of topological vector
spaces}\label{sec:5}

In this section, we introduce two important classes of inductive
limits of locally convex spaces. It
turns out that these classes play an essential role in our analysis of
extensions of time-varying real analytic vector fields

\begin{definition}
Let $\{V_i,f_{i}\}_{i\in \N}$ be an inductive family of locally convex
spaces and the pair $(V,\{g_i\}_{i\in \N})$ be
the locally convex inductive limit of $\{V_i,
f_{i}\}_{i\in \N}$. The inductive family $\{V_i, f_{i}\}_{i\in \N}$ is \textbf{regular} if, for every bounded set
$B\subset V$, there exists $m\in \N$ and a bounded set $B_m\subset
V_m$ such that the restriction map $g_m\mid_{B_m}:B_m\to V$ is a
bijection onto $B$.

The inductive family $\{V_i, f_{i}\}_{i\in \N}$ is \textbf{boundedly
  retractive} if, for every bounded set $B\subset V$, there exists
$m\in \N$ and a bounded set $B_m\subset V_m$ such that the restriction
map $g_m\mid_{B_m}:B_m\to V$ is a homeomorphism onto $B$.
\end{definition}

While most of the well-known inductive family of locally convex spaces in mathematics are
regular and/or boundedly retractive, checking
whether an inductive family is 
regular or boundedly retractive using the definitions is very difficult. However, some properties of the
connecting maps of the inductive family can ensure that the inductive
limit is regular or boundedly retractive. 

\begin{definition}
Let $\{V_i\}_{i\in \N}$ be a family of locally convex topological
vector spaces and let $\{f_{i}\}_{i\in \N}$ be a family of
continuous linear maps
such that $f_{i}:V_i\to V_{i+1}$.
\begin{enumerate} 

\item The inductive family $\{V_i, f_{i}\}_{i\in \N}$ is \textbf{compact} if, for every $i\in \N$,
the map $f_{i}:V_i\to V_{i+1}$ is compact.

\item The inductive family $\{V_i, f_{i}\}_{i\in \N}$ is \textbf{weakly compact} if, for every $i\in \N$,
the map $f_{i}:V_i\to V_{i+1}$ is weakly compact.
\end{enumerate}
\end{definition}

In order to study the compactness (weak compactness) of an inductive family of locally convex spaces $\{V_i, f_{i}\}_{i\in \N}$, it is essential that one can characterize
the compact (weakly compact) subsets of locally convex vector spaces
$V_i$ for every $i\in \N$.
For a metrizable topological vector space $V$, it is well-known that a set
$K\subseteq X$ is compact if and only if every sequence in $K$ has a
convergent subsequence. However, it is possible that the weak
topology on $V$ is not metrizable. Thus it would be interesting to see
if the same characterization holds for weakly compact subsets of
$V$. Eberlein\textendash Smulian Theorem answers this question
affirmatively for Banach spaces
\cite[Chapter IV, Corollary 2]{schaefer}. 
\begin{theorem}
Let $V$ be a Banach space and $A\subseteq V$. Then the following
statements are equivalent:
\begin{enumerate}
\item[(i)] The weak closure of $A$ is weakly compact,  
\item[(ii)] each sequence of elements of $A$ has a subsequence that is
  weakly convergent. 
\end{enumerate}
\end{theorem}

One can get a partial generalization of the Eberlein\textendash Smulian Theorem for complete locally
convex spaces \cite[Chapter IV, Theorem 11.2]{schaefer}. 

\begin{theorem}\label{th:52}
Let $V$ be a complete locally convex space and $A\subseteq V$. If
every sequence of elements of $A$ has a subsequence that is weakly
convergent, then the weak closure of $A$ is weakly compact.
\end{theorem}

The next theorem shows that an inductive family of locally convex spaces with compact (weakly
compact) connecting maps is boundedly retractive (regular).

\begin{theorem}\label{th:25}
Let $\{V_i\}_{i\in \N}$ be a family of locally convex topological
vector spaces and let $\{f_{i}\}_{i\in \N}$ be a family of linear
continuous maps
such that $f_{i}:V_i\to V_{i+1}$. Then 
\begin{enumerate}
\item if the inductive family $\{V_i, f_{i}\}_{i\in \N}$ is weakly
  compact, then it is regular, and
\item if the inductive family $\{V_i, f_{i}\}_{i\in \N}$ is
  compact, then it is boundedly retractive. 
\end{enumerate}
\end{theorem}
\begin{proof}
The first part of this theorem has been proved in \cite[Theorem 6]{HK:1967} and
the second part in \cite[Theorem 6']{HK:1967}
\end{proof}
However, one can find boundedly retractive inductive families which are
not compact \cite{KDB:1988}. In \cite{VSR:1970}, Retakh studied an important condition on
inductive families of locally convex spaces called condition $(M)$. 
\begin{definition}
Let $\{V_i\}_{i\in \N}$ be a family of locally convex topological
vector spaces and let $\{f_{i}\}_{i\in \N}$ be a family of linear
continuous maps such that $f_{i}:V_i\to V_{i+1}$. The inductive family $\{V_i,
f_{i}\}_{i\in \N}$ satisfies \textbf{condition ($M$)} if there
exists a sequence of absolutely convex neighbourhoods $\{U_i\}_{i\in
  \N}$ of $0$ such that, for every $i\in \N$, we have $U_i\subseteq
V_i$ and,
\begin{enumerate}
\item for every $i\in \N$, we have $U_i\subseteq f^{-1}_{i}(U_{i+1})$, and
\item for every $i\in \N$, there exists $M_i>0$ such that, for every
  $j>M_i$, the topologies induced from $V_j$ on $U_i$ are all the same.  
\end{enumerate}
\end{definition}

It can be shown that condition $(M)$ has close connection with
regularity of inductive families of locally convex spaces \cite{KDB:1988}. 

\begin{theorem}\label{th:24}
Let $\{V_i\}_{i\in \N}$ be a family of normed
vector spaces and let $\{f_{i}\}_{i\in \N}$ be a family of 
continuous linear maps
such that $f_{i}:V_i\to V_{i+1}$. Suppose that the inductive family
$\{V_i, f_{i}\}_{i\in \N}$ is regular. Then inductive family $\{V_i, f_{i}\}_{i\in \N}$ is boundedly retractive if and only if it
satisfies condition ($M$).
\end{theorem}
\begin{proof}
This theorem is proved in \cite[Proposition 9(d)]{KDB:1988}.
\end{proof}

\section{Time-varying vector fields and their flows}\label{sec:6}

In this section, we define and study time-varying $C^{\nu}$-vector
field. 

\begin{definition}
Let $M$ be a $C^{\nu}$-manifold and $\mathbb{T}\subseteq \R$ be an
interval. Then a map $X:\mathbb{T}\times M\to TM$ is a
\textbf{time-varying $C^{\nu}$-vector field} if, for every $t\in
\mathbb{T}$, the map $X^t:M\to TM$ defined as
\begin{equation*}
X^t(x)=X(t,x),\qquad\forall x\in M,
\end{equation*}
is a $C^{\nu}$-vector field.
\end{definition}

Associated to every time-varying $C^{\nu}$-vector
field $X:\mathbb{T}\times M\to TM$, one can define a curve
$\widehat{X}:\mathbb{T}\to\Gamma^{\nu}(TM)$ such that
\begin{equation*}
\widehat{X}(t)(x)=X(t,x),\qquad\forall t\in \mathbb{T},\ \forall x\in M. 
\end{equation*}
It is clear that this correspondence between time-varying
$C^{\nu}$-vector fields and curves on the space $\Gamma^{\nu}(TM)$ is one-to-one.

In order to study properties of time-varying $C^{\nu}$-vector fields, we
need to define a topology on the space $\Gamma^{\nu}(TM)$. In the
holomorphic case, the natural topology on the space
$\Gamma^{\hol}(TM)$ is the 
so-called ``compact-open'' topology, which has been throughly studied in the literature \cite[\S 8]{kriegl1997convenient}. 

\begin{definition}
Let $K\subseteq M$ be a compact set. Then we define the seminorm
$p^{\hol}_K$ on $\Gamma^{\hol}(TM)$ by
\begin{equation*}
p^{\hol}_{K}(X)=\left\{\left\|X(x)\right\|\suchthat x\in K\right\}
\end{equation*}
The family of seminorms $\{p^{\hol}_{K}\}$ define a locally convex
topology on $\Gamma^{\hol}(TM)$ called the \textbf{$C^{\hol}$-topology}.
\end{definition}

Properties of $C^{\hol}$-topology on $\Gamma^{\hol}(TM)$ has been
investigated in \cite[\S]{kriegl1997convenient}. The following theorem has been proved in \cite[\S8.4]{kriegl1997convenient}.

\begin{theorem}
The vector space $\Gamma^{\hol}(TM)$ equipped with the
$C^{\hol}$-topology is a Hausdorff, separable, complete, metrizable,
and nuclear locally convex space.
\end{theorem}

In the real analytic case, it is natural to equip $\Gamma^{\omega}(TM)$ with
the subspace topology from $\Gamma^{\infty}(TM)$. However, it can be shown that
this topology on $\Gamma^{\omega}(TM)$ is not complete \cite[Chapter 5]{SJ-ADL:2014}. Another
topology on $\Gamma^{\omega}(TM)$ can be defined using the fact that,
every real analytic vector field is the germ of a holomorphic vector
field, defined on a suitable domain. We will see that this topology on $\Gamma^{\omega}(TM)$ makes it
into a complete, separable, and nuclear space. Each
of these properties is essential for validity of our extension
results. In \cite{AM:1966}, using the so-called compact-open topology
on space of holomorphic functions, two characterization for a topology
on the space of real analytic functions has been developed. This
topology on the space $C^{\omega}(M)$ has been further studied in
\cite{PD:2010}. In this section, using the same setting as in
\cite{AM:1966}, we define a topology on the space of real analytic
functions.

 While Two different characterization of
this topology has been studied in. 

Let $M$ be a real analytic manifold and $M^{\C}$ be a complexification
of $M$. We denote the set of all holomorphic vector fields on
$\overline{U}$ by $\Gamma^{\hol}(T\overline{U})$. We define
$\Gamma^{\hol,\R}(T\overline{U})\subseteq
\Gamma^{\hol}(T\overline{U})$ as 
\begin{equation*}
\Gamma^{\hol,\R}(T\overline{U})=\left\{X\in
  \Gamma^{\hol}(T\overline{U})\suchthat X(x)\in T_xM, \forall x\in M\right\}
\end{equation*}

Then, for every neighbourhood $\overline{U}\subseteq M^{\C}$
containing $M$, we define the map
$i^{\R}_{\overline{U}}:\Gamma^{\hol,\R}(T\overline{U})\to
\Gamma^{\omega}(TM)$ as
\begin{equation*}
i^{\R}_{\overline{U}}(X)=X\mid_{M}.
\end{equation*}

If we denote the set of all the neighbourhoods $\overline{U}\in M^{\C}$
of $M$ by $\mathscr{N}_M$. Then we can define the inductive limit
topology on $\Gamma^{\omega}(TM)$.

\begin{definition}
The \textbf{inductive topology} on $\Gamma^{\omega}(TM)$ is defined
as the finest locally convex topology which makes all the maps
$\{i^{\R}_{\overline{U}}\}_{\overline{U}\in\mathscr{N}_M}$ continuous.
\end{definition}

Although the definition of inductive topology on $\Gamma^{\omega}(TM)$ is natural, characterization
of properties of $\Gamma^{\omega}(TM)$ using this topology is not easy. The
main reason is that, for
non-compact $M$, the inductive limit $\varinjlim_{\overline{U}\in\mathscr{N}_M} \Gamma^{\hol,\R}(T\overline{U})=\Gamma^{\omega}(TM)$ is not countable
\cite[Fact 14]{PD:2010}. However, one can define another topology on
the space of
real analytic sections which is representable by countable inductive
and projective limits \cite{AM:1966}. 

Let $K\subseteq M$ be a compact set and $\mathscr{N}_K$ be the set of
all neighbourhoods of $K$ in $M^{\C}$. Then we denote the space of
germs of holomorphic vector fields around $K$ by
$\mathscr{G}^{\hol}_K$. In other words, we have
\begin{equation*}
\varinjlim \Gamma^{\hol}(T\overline{U})=\mathscr{G}^{\hol}_{K},
\end{equation*}
where the inductive limit is on the directed set $\mathscr{N}_K$. One
can equip the space $\mathscr{G}^{\hol}_K$ with the 
locally convex topology defined using the above inductive limit.

It turns out that $\mathscr{G}^{\hol}_{K}$ can also be expressed as a
inductive limit of a countable family of Banach spaces
\cite{PD:2010}. Note that, for every compact set $K\subseteq M$, one can choose a
sequence of open sets $\{\overline{U}_n\}_{n\in\N}$ in $M^{\C}$ such that, for every $n\in \N$, we have
\begin{equation*}
\mathrm{cl}(\overline{U}_{n+1})\subseteq \overline{U}_{n},
\end{equation*}
and $\bigcap_{i=1}^{\infty} \overline{U}_i=K$. Then we have $\varinjlim_{n\to\infty}
\Gamma^{\hol}(T\overline{U}_n)=\mathscr{G}^{\hol}_{K}$. 
\begin{definition}
Let $\overline{U}$ be an open set in $M^{\C}$. We define the map
$p_{\overline{U}}:\Gamma^{\hol}(T\overline{U})\to [0,\infty]$ by
\begin{equation*}
p_{\overline{U}}(X)=\sup\{\|X(x)\|\mid x\in \overline{U}\},\qquad\forall
X\in \Gamma^{\hol}(\overline{U}).
\end{equation*}
Then $\Gamma^{\hol}_{\bdd}(T\overline{U})$ is a subspace of $\Gamma^{\hol}(T\overline{U})$ defined as
\begin{equation*}
\Gamma^{\hol}_{\bdd}(T\overline{U})=\{X\in \Gamma^{\hol}(T\overline{U})\mid
p_{\overline{U}}(X)< \infty\}.
\end{equation*}
We equip $\Gamma^{\hol}_{\bdd}(T\overline{U})$ with the norm
$p_{\overline{U}}$ and define the inclusion
$\rho_{\overline{U}}:\Gamma^{\hol}_{\bdd}(T\overline{U})\to
\Gamma^{\hol}(T\overline{U})$ as
\begin{equation*}
\rho_{\overline{U}}(X)=X,\qquad\forall X\in \Gamma^{\hol}_{\bdd}(T\overline{U}).
\end{equation*}
\end{definition}

\begin{theorem}\label{th:17}
The space $(\Gamma^{\hol}_{\bdd}(T\overline{U}),p_{\overline{U}})$ is a
Banach space and the map
$\rho_{\overline{U}}:\Gamma^{\hol}_{\bdd}(T\overline{U})\to \Gamma^{\hol}(T\overline{U})$ is a compact continuous map. 
\end{theorem} 
\begin{proof}
Let $K$ be a compact subset of $M\cap \overline{U}$. Then, for every $X\in
\Gamma^{\hol}_{\bdd}(T\overline{U})$, we have
$p^{\hol}_K(\rho_{\overline{U}}(X))=p^{\hol}_{K}(X)\le p_{\overline{U}}(X)$, which
  implies that $\rho_{\overline{U}}$ is continuous. Now consider the
  open set $p_{\overline{U}}^{-1}\left([0,1)\right)$ in
  $\Gamma^{\hol}_{\bdd}(T\overline{U})$. The set $p_{\overline{U}}^{-1}\left([0,1)\right)$
  is bounded and $\rho_{\overline{U}}$ is continuous. So 
\begin{equation*}
\rho_{\overline{U}}\left(p_{\overline{U}}^{-1}\left([0,1)\right)\right),
\end{equation*}
is bounded in $\Gamma^{\hol}(T\overline{U})$. Since
$\Gamma^{\hol}(T\overline{U})$ is nuclear, it satisfies the
Heine\textendash Borel property \cite[Chapter III, \S 7]{schaefer}. Thus, the bounded the set
$\rho_{\overline{U}}\left(p_{\overline{U}}^{-1}\left([0,1)\right)\right)$
is relatively compact in $\Gamma^{\hol}(T\overline{U})$. So
$\rho_{\overline{U}}$ is compact.

Now we show that $(\Gamma^{\hol}_{\bdd}(T\overline{U}),p_{\overline{U}})$ is a Banach space. Let $\{X_n\}_{n\in\N}$ be a Cauchy sequence in
$\Gamma^{\hol}_{\bdd}(T\overline{U})$. It suffices to show that there exists $X\in
\Gamma^{\hol}_{\bdd}(T\overline{U})$ such that $\lim_{n\to\infty} X_n=X$ in the
topology induced by $p_{\overline{U}}$ on $\Gamma^{\hol}_{\bdd}(T\overline{U})$. Since $\rho_{\overline{U}}$ is
continuous, the sequence $\{X_n\}_{n\in\N}$ is Cauchy in
$\Gamma^{\hol}(T\overline{U})$. Since $\Gamma^{\hol}(T\overline{U})$ is complete,
there exists $X\in \Gamma^{\hol}(T\overline{U})$ such that
$\lim_{n\to\infty} X_n=X$ in the $C^{\hol}$-topology. Now we show
that $\lim_{n\to\infty} X_n=X$ in the topology of $(\Gamma^{\hol}_{\bdd}(T\overline{U}),p_{\overline{U}})$ and $X\in
\Gamma^{\hol}_{\bdd}(T\overline{U})$. 
Let $\epsilon>0$. Then there exists $N\in \N$ such that, for every $n,m>N$, we have
\begin{equation*}
p_{\overline{U}}(X_n-X_m)<\frac{\epsilon}{2}.
\end{equation*}
This implies that, for every $z\in U$ and every $n,m>N$, we have
\begin{equation*}
\|X_n(z)-X_m(z)\|<\frac{\epsilon}{2}.
\end{equation*}
So, for every $z\in U$ and every $n>N$, we choose $m_z>N$ such that 
\begin{equation*}
\|X_m(z)-X(z)\|<\frac{\epsilon}{2},\qquad\forall m\ge m_z.
\end{equation*}
This implies that, for every $z\in U$, we have 
\begin{equation*}
\|X(z)-X_n(z)\|<\|X_n(z)-X_{m_z}(z)\|+\|X_{m_z}(z)-X(z)\|<\epsilon.
\end{equation*}
So, for every $n>N$, we have 
\begin{equation*}
p_{\overline{U}}(X_n-X)<\epsilon.
\end{equation*}
This completes the proof.
\end{proof}

\begin{theorem}\label{th:2}
Let $K$ be a compact set and $\{\overline{U}_n\}_{n\in\N}$ be a
sequence of open, relatively compact neighbourhoods of $K$ in $M^{\C}$ such that
\begin{equation*}
\mathrm{cl}(\overline{U}_{n+1})\subseteq \overline{U}_n,\qquad\forall n\in \N,
\end{equation*}
and $\bigcap_{n\in \N} \overline{U}_n=K$. Then we have $\varinjlim_{n\to\infty} \Gamma^{\hol}_{\bdd}(T\overline{U}_n)=\mathscr{G}^{\hol}_K$. Moreover, the
inductive limit is compact.
\end{theorem}
\begin{proof}
For every $n\in \N$, we define $r_n:\Gamma^{\hol}(T\overline{U}_n)\to
\Gamma^{\hol}_{\bdd}(T\overline{U}_{n+1})$ as 
\begin{equation*}
r_n(X)=X\mid_{\overline{U}_{n+1}},\qquad\forall X\in \Gamma^{\hol}(T\overline{U}_n).
\end{equation*}
For every compact set $C$ with $\overline{U}_{n+1}\subseteq C\subseteq
\overline{U}_{n}$, we have $p_{\overline{U}_{n+1}}(X) \le p^{\hol}_{C}(X)$. This implies
that the map $r_n$ is continuous and we have the following diagram:
\begin{equation*}
\xymatrix{ \Gamma^{\hol}_{\bdd}(T\overline{U}_{n}) \ar[r]^{\rho_{{U}_n}}&
  \Gamma^{\hol}(T\overline{U}_{n})\ar[r]^{r_n}&
  \Gamma^{\hol}_{\bdd}(T\overline{U}_{n+1})\ar[r]^{\rho_{{U}_{n+1}}}& \Gamma^{\hol}(T\overline{U}_{n+1})}.
\end{equation*}
Since all maps in the above diagram are linear and continuous, by the
universal property of the inductive limit of locally convex spaces, we have  
\begin{equation*}
\varinjlim_{n\to\infty} \Gamma^{\hol}_{\bdd}(T\overline{U}_n)=\varinjlim_{n\to\infty}
\Gamma^{\hol}(T\overline{U}_n)=\mathscr{G}^{\hol}_K.
\end{equation*}
Moreover, for every $n\in \N$, the map $\rho_{\overline{U}_n}$ is compact and $r_n$ is
continuous. So the composition $r_n\scirc\rho_{\overline{U}_n}$ is also
compact \cite[\S 17.1, Proposition 1]{HJ:1981}. This implies that the
direct limit 
\begin{equation*}
\varinjlim_{n\to\infty} \Gamma^{\hol}_{\bdd}(T\overline{U}_n)=\mathscr{G}^{\hol}_K
\end{equation*}
is compact.
\end{proof}

One can define the subspace $\mathscr{G}^{\hol,\R}_{K}\subseteq
\mathscr{G}^{\hol}_{K}$ as
\begin{equation*}
\mathscr{G}^{\hol,\R}_{K}=\left\{[X]_K\suchthat \exists
  \overline{U}\in \mathscr{N}_K,\  X\in \Gamma^{\hol,\R}(T\overline{U})\right\}
\end{equation*}
Let $\{K_n\}_{n\in \N}$ be a compact exhaustion for $M$. Then we have 
\begin{equation*}
\varprojlim \mathscr{G}^{\hol,\R}_{K_n}=\Gamma^{\omega}(TM).
\end{equation*}
Using this projective limit, one can define another topology on space
of real analytic vector fields.

\begin{definition}
Let $\{K_n\}_{n\in \N}$ be a compact exhaustion for $M$. Then we
define the \textbf{projective limit topology} on $\Gamma^{\omega}(TM)$ as the projective limit
topology defined using the following projective family of locally
convex spaces: 
\begin{equation*}
\varprojlim \mathscr{G}^{\hol,\R}_{K_n}=\Gamma^{\omega}(TM).
\end{equation*}
It is easy to show that the projective limit topology on
$\Gamma^{\omega}(TM)$ does not depend on a specific choice of the
compact exhaustion $\{K_n\}_{n\in \N} $ for $M$.
\end{definition}

It is a deep theorem of Martineau that the projective limit topology
and inductive limit topology on $\Gamma^{\omega}(TM)$ coincide
\cite{AM:1966}. We denote this topology on $\Gamma^{\omega}(TM)$ by
the $C^{\omega}$-topology. One can show that this topology has nice
properties \cite[\S 5.3]{SJ-ADL:2014}

\begin{theorem}\label{th:60}
The vector space $\Gamma^{\omega}(TM)$ equipped with the
$C^{\omega}$-topology is a Hausdorff, separable, complete, and nuclear locally
convex space. 
\end{theorem}

As is shown in Theorem \ref{th:9} the real analytic vector fields are exactly the derivations of the $\R$-algebra
$C^{\omega}(M)$. Since derivations of $C^{\omega}(M)$ are linear mappings from
$C^{\omega}(M)$ to $C^{\omega}(M)$, it would be interesting to study
the more general space of linear mapping from $C^{\omega}(N)$ to $C^{\omega}(M)$. 

\begin{definition}
Let $M$ and $N$ be real analytic manifolds. The space of linear mapping
from $C^{\omega}(N)$ to $C^{\omega}(M)$ is denoted by $\mathrm{L}(C^{\omega}(N);C^{\omega}(M))$. 
\end{definition}

One can define different topologies on
$\mathrm{L}(C^{\omega}(N);C^{\omega}(M))$, using the
$C^{\omega}$-topologies on the spaces $C^{\omega}(M)$ and
$C^{\omega}(N)$. In this section, we focus on the topology of
pointwise convergence on $\mathrm{L}(C^{\omega}(N);C^{\omega}(M))$. We
will see that this topology is consistent with the $C^{\omega}$-topology on $\Gamma^{\omega}(TM)$.

\begin{definition}
For $f\in C^{\omega}(M)$, we define the map $\mathscr{L}_{f}:\LC{M}{N}{\omega}\to
C^{\omega}(N)$ as 
\begin{equation*}
\mathscr{L}_{f}(X)=X(f).
\end{equation*}
The \textbf{topology of pointwise convergence} on $\LC{M}{N}{\omega}$ is the projective topology with
respect to the family $\{C^{\omega}(N),\mathscr{L}_f\}_{f\in C^{\omega}(M)}$.
\end{definition}

It can be shown that
$\mathrm{L}(C^{\omega}(N);C^{\omega}(M))$ equipped with the topology of pointwise convergence has many nice properties. 

\begin{theorem}
The vector space $\mathrm{L}(C^{\omega}(N);C^{\omega}(M))$ with the
topology of pointwise convergence is a Hausdorff, separable, complete, and nuclear
locally convex space.
\end{theorem}  
\begin{proof}
We show that $\LC{M}{N}{\omega}$ is a closed subspace of
$C^{\omega}(N)^{C^{\omega}(M)}$, if we equip the latter space with its
natural topology of pointwise convergence. Let $\{X_{\alpha}\}_{\alpha\in \Lambda}$ be
a converging net in $\LC{M}{N}{\omega}$ with the limit $X\in
C^{\omega}(N)^{C^{\omega}(M)}$. We show that $X$ is linear. Let $f,g\in C^{\omega}(M)$ and $c\in \mathbb{F}$. Then we have 
\begin{equation*}
X_{\alpha}(f+cg)=X_{\alpha}(f)+cX_{\alpha}(g),\qquad\forall \alpha\in \Lambda.
\end{equation*} 
By taking limit on $\alpha$, we get
\begin{equation*}
X(f+cg)=X(f)+cX(g).
\end{equation*} 
This implies that $X$ is linear and therefore
$\mathrm{L}(C^{\omega}(N);C^{\omega}(M))$ is a closed subspace of
$C^{\omega}(N)^{C^{\omega}(M)}$.

Since $C^{\omega}(N)$ is Hausdorff, it is clear that
$C^{\omega}(N)^{C^{\omega}(M)}$ is Hausdorff. This implies that
$\LC{M}{N}{\omega}\subseteq C^{\omega}(N)^{C^{\omega}(M)}$ is Hausdorff.
Let $\mathbf{c}$ be the cardinality of the continuum. Note
that $C^{\omega}(M)\subseteq C^{0}(M)$ and $M$ is second countable and
hence separable. This implies
that the cardinality of $C^{0}(M)$
is $\mathbf{c}$ \cite[Chapter 5, Theorem 2.6(a)]{KH-JT:1999}. Therefore, the cardinality of $C^{\omega}(M)$ is at most $\mathbf{c}$. The
product of $\mathbf{c}$ separable spaces is separable
\cite[Theorem 16.4(c)]{SW:2004}. This implies that
$C^{\omega}(N)^{C^{\omega}(M)}$ is separable. Since $\LC{M}{N}{\omega}$ is a
closed subspace of $C^{\omega}(N)^{C^{\omega}(M)}$, it is separable
\cite[Theorem 16.4]{SW:2004}. Note that $C^{\omega}(N)$ is complete. This implies that
$C^{\omega}(N)^{C^{\omega}(M)}$ is complete \cite[Chapter II, \S5.3]{schaefer}. Since $\LC{M}{N}{\omega}$ is a closed
subspace of $C^{\omega}(N)^{C^{\omega}(M)}$, it is complete. The product of
any arbitrary family of nuclear locally convex vector spaces
is nuclear \cite[Chapter III, \S7.4]{schaefer}. This implies that $C^{\omega}(N)^{C^{\omega}(M)}$ is
nuclear. Since every closed subspace of nuclear space is nuclear
\cite[Chapter III, \S7.4]{schaefer},
$\LC{M}{N}{\omega}$ is also nuclear. 
\end{proof}

We have already mentioned that real analytic vector fields on
$M$ are exactly derivations on $\Gamma^{\omega}(TM)$. Thus, we have
\begin{equation*}
\Gamma^{\omega}(TM)\subseteq \mathrm{L}(C^{\omega}(N);C^{\omega}(M)).
\end{equation*}
Therefore, the topology of pointwise convergence on
$\mathrm{L}(C^{\omega}(N);C^{\omega}(M))$ will induce a subspace
topology on $\Gamma^{\omega}(TM)$. It is interesting to note that this
subspace topology on $\Gamma^{\omega}(TM)$ and the $C^{\omega}$-topology on
$\Gamma^{\omega}(TM)$ are the same \cite[Theorem 5.8]{SJ-ADL:2014}.

\begin{theorem}
The $C^{\omega}$-topology on $\Gamma^{\omega}(TM)$ coincides with the
subspace topology form $\mathrm{L}(C^{\omega}(N);C^{\omega}(M))$.
\end{theorem}

Thus, it is reasonable to denote the topology of pointwise convergence on
$\mathrm{L}(C^{\omega}(N);C^{\omega}(M))$ by the $C^{\omega}$-topology.

It is well-known every locally convex topology can be
characterized using a family of seminorms \cite[Theorem
1.37]{rudinfunctional}.  Since the vector space $\Gamma^{\omega}(TM)$
equipped with the $C^{\omega}$-topology is a locally
convex space, it would be interesting to provide an explicit family of
seminorm for the locally convex space $\Gamma^{\omega}(TM)$. As to our
knowledge, the first characterization of the space of germs of holomorphic
functions on compact subsets of $\C^n$ using an explicit family of
seminorms has been developed in \cite{JM:1984}. In the notes \cite{PD:2010}, a family of seminorm on
$\Gamma^{\omega}(TM)$ has been introduced and it has been mentioned
that the $C^{\omega}$-topology on $C^{\omega}(M)$ is generated by this family of seminorms. For the case $M=\R$, the complete
proof of the fact that this family of seminorms generates the
$C^{\omega}$-topology on $C^{\omega}(\R)$ has been
given in \cite{DV:2013}. Using the idea of the
proof in \cite{DV:2013}, a complete characterization of the locally
convex space $\Gamma^{\omega}(TM)$ using a family of seminorm has been
given in \cite{SJ-ADL:2014}. In this section, we provide a family of
seminorms for the $C^{\omega}$-topology on the space $\mathrm{L}(C^{\omega}(M);
C^{\omega}(N))$. Since $\Gamma^{\omega}(TM)$ can be considered as a
subspace of  $\mathrm{L}(C^{\omega}(M);
C^{\omega}(M))$, this family of seminorms also gives a family of
generating seminorms for the $C^{\omega}$-topology on $\Gamma^{\omega}(TM)$. 

\begin{definition}
Let $\mathbf{c}^{\downarrow}_{0}(\Z_{\ge 0},\R_{>0}, d)$ denote the
set of all decreasing sequences $\{a_n\}_{n\in \Z_{\ge 0}}$ such that,
for every $n\in \Z_{\ge 0}$, we have $0<a_n\le d$ and 
\begin{equation*}
\lim_{n\to\infty} a_n=0.
\end{equation*}
\end{definition}
\begin{definition}
Let $U$ be a coordinate chart on $N$, $K\subseteq U$ be a compact
set, $\mathbf{a}\in \mathbf{c}^{\downarrow}_{0}(\Z_{\ge 0},\R_{>0},
d)$, and $f\in C^{\omega}(M)$. Then, for every $X\in
\mathrm{L}(C^{\omega}(M); C^{\omega}(N))$, we define 
\begin{equation*}
p^{\omega}_{K,\mathbf{a},f}(X)=\left\{\frac{a_0a_1\ldots
    a_{|r|}}{|(r)|!}\left\|D^{(r)}Xf(x)\right\|\suchthat |(r)|\in \Z_{\ge
    0}, \ x\in K\right\}
\end{equation*}
\end{definition}

Using \cite[Theorem 5.5]{SJ-ADL:2014}, we have

\begin{theorem}
The family of seminorms $\{p^{\omega}_{K,\mathbf{a},f}\}$ generates
the $C^{\omega}$-topology on $\mathrm{L}(C^{\omega}(M); C^{\omega}(N))$
\end{theorem}

Now, we prove a specific
approximation for the seminorms on $\Gamma^{\omega}(M)$. In section \ref{sec:10},
we will see that this approximation is useful in
studying flows of time-varying real analytic vector fields. Let $d>0$ be a positive real number and $\mathbf{a}\in \mathbf{c}^{\downarrow}_{0}(\Z_{\ge 0},\R_{>0}, d)$. For every $n\in
\N$, we define the sequence
$\mathbf{a}_n=(a_{n,0},a_{n,1},\ldots,a_{n,m},\ldots)$ as 
\begin{equation*}
a_{n,m}=
\begin{cases}
\left(\frac{m+1}{m}\right)^na_m,& m>n,\\
\left(\frac{m+1}{m}\right)^ma_{m},& m\le n.
\end{cases}
\end{equation*}
Associated to every $\mathbf{a}\in \mathbf{c}^{\downarrow}_{0}(\Z_{\ge 0},\R_{>0}, d)$,
we define the sequence $\mathbf{b}_n\in \mathbf{c}^{\downarrow}_{0}(\Z_{\ge 0}, \R_{>0})$ as
\begin{equation*}
b_{n,m}=
\begin{cases}
a_{n,m}, & m=0,m=1,\\
\left(\frac{(m+1)(m+2)}{(m-1)(m)}\right)a_{n,m},& m>1.
\end{cases}
\end{equation*} 
\begin{lemma}\label{lem:4}
Let $\mathbf{a}\in \mathbf{c}^{\downarrow}_{0}(\Z_{\ge 0},\R_{>0},d)$. Then, for every
$n\in \Z_{\ge 0}$, we have $\mathbf{a}_n\in \mathbf{c}^{\downarrow}_{0}(\Z_{\ge
  0},\R_{>0},\textup{e}d)$ and, for every $m,n\in \Z_{\ge 0}$, we have
\begin{eqnarray*}
a_{n,m}&\le& \textup{e}a_m,\\
\frac{(m+1)}{(n+1)}&\le&\frac{(a_{n+1,0})(a_{n+1,1})\ldots
                                    (a_{n+1,m+1})}{(a_{n,0})(a_{n,1})\ldots
                                    (a_{n,m+1})},
\end{eqnarray*}
where $\textup{e}$ is the Euler constant. Moreover, for every $n\in \Z_{\ge 0}$ we have $\mathbf{b}_n\in
\mathbf{c}^{\downarrow}_{0}(\Z_{\ge 0}, \R_{>0},6\textup{e}d)$ and, for every $m>1$, we have 
\begin{eqnarray*}
b_{n,m}&\le & 6\textup{e}a_m,\\
\frac{(a_{n,0})(a_{n,1})\ldots (a_{n,m})}{(m-2)!}&=&\frac{(b_{n,0})(b_{n,1})\ldots (b_{n,m})}{m!}.
\end{eqnarray*}
\end{lemma}
\begin{proof}
Let $\mathbf{a}\in \mathbf{c}^{\downarrow}_{0}(\Z_{\ge 0},\R_{>0},d)$. Then by
definition of $\mathbf{a}_n$, for $n<m$, we have 
\begin{equation*}
a_{n,m}=\left(\frac{m+1}{m}\right)^na_m\le
\left(\frac{m+1}{m}\right)^ma_m\le \textup{e}a_m
\end{equation*}
For $n\ge m$, we have
\begin{equation*}
a_{n,m}=\left(\frac{m+1}{m}\right)^ma_m\le \textup{e}a_m.
\end{equation*}
This implies that $\lim_{m\to\infty} a_{n,m}=0$. Moreover, for every
$m,n\in\Z_{\ge 0}$, we have 
\begin{equation*}
a_{n,m}\le \textup{e}a_m\le \textup{e}d.
\end{equation*}
So we have $\mathbf{a}_n\in \mathbf{c}^{\downarrow}_{0}(\Z_{\ge
  0},\R_{>0},\textup{e}d)$. 
Let $m,n\in \Z_{\ge 0}$ be such that $n+1>m+1$. Then we have 
\begin{equation*}
\frac{a_{n+1,m+1}}{a_{n,m+1}}=1.
\end{equation*}
So we get 
\begin{equation*}
\frac{(a_{n+1,0})(a_{n+1,1})\ldots (a_{n+1,m+1})}{(a_{n,0})(a_{n,1})\ldots
  (a_{n,m+1})}\ge 1.
\end{equation*}
Since we have $\mathbf{a}_n\in \mathbf{c}^{\downarrow}_{0}(\Z_{\ge 0},\R_{>0}, \textup{e}d)$, we
get 
\begin{equation*}
\frac{(a_{n+1,0})(a_{n+1,1})\ldots
  (a_{n+1,m+1})}{(a_{n,0})(a_{n,1})\ldots (a_{n,m+1})}\ge 1\ge \frac{m+1}{n+1}.
\end{equation*}
Now suppose that $m,n\in \Z_{\ge 0}$ are such that $n+1\le m+1$. Then
we have
\begin{equation*}
\frac{a_{n+1,m+1}}{a_{n,m+1}}=\left(\frac{m+1}{m}\right).
\end{equation*}
Therefore, we get 
\begin{equation*}
\frac{(a_{n+1,0})(a_{n+1,1})\ldots (a_{n+1,m+1})}{(a_{n,0})(a_{n,1})\ldots
  (a_{n,m+1})}=\left(\frac{n+2}{n+1}\right)
\left(\frac{n+3}{n+2}\right)\ldots
\left(\frac{m+2}{m+1}\right)=\frac{m+2}{n+1}>\frac{m+1}{n+1}.
\end{equation*}
Since we have $\mathbf{a}_n\in \mathbf{c}^{\downarrow}_{0}(\Z_{\ge 0},\R_{>0}, \textup{e}d)$, we
get 
\begin{equation*}
\frac{(a_{n+1,0})(a_{n+1,1})\ldots (a_{n+1,m+1})}{(a_{n,0})(a_{n,1})\ldots
  (a_{n,m})}\ge \frac{m+1}{n+1}.
\end{equation*}
So, for all $m,n\in \Z_{\ge 0}$, we have
\begin{equation*}
\frac{(a_{n+1,0})(a_{n+1,1})\ldots (a_{n+1,m+1})}{(a_{n,0})(a_{n,1})\ldots
  (a_{n,m+1})}\ge \frac{m+1}{n+1}.
\end{equation*}
Finally, since $\mathbf{a}_n\in \mathbf{c}^{\downarrow}_{0}(\Z_{\ge
  0},\R_{>0},\textup{e}d)$ and we have $\frac{(m+2)(m+1)}{m(m-1)}\le
6$, for all $m>1$, we get
\begin{equation*}
b_{n,m}=\frac{(m+2)(m+1)}{m(m-1)}a_{n,m}\le 6a_{n,m}.
\end{equation*}
So we have $\lim_{m\to\infty} b_{n,m}=6\lim_{m\to\infty}
a_{n,m}=0$. Moreover, we have 
\begin{equation*}
b_{n,m}\le 6a_{n,m}\le 6\textup{e}a_m\le 6\textup{e}d.
\end{equation*}
Thus we get $\mathbf{b}_n\in \mathbf{c}^{\downarrow}_{0}(\Z_{\ge
  0},\R_{>0},6\textup{e}{d})$. This completes the proof of the lemma.
\end{proof}
\begin{theorem}\label{th:54}
Let $M$ be a real analytic manifold of dimension $N$, $X\in
\Gamma^{\omega}(TM)$, and $f\in C^{\omega}(M)$. Let $U$ be a
coordinate neighbourhood in $M$ and $K\subseteq U$ be compact. For every $d>0$,
every $\mathbf{a}\in \mathbf{c}^{\downarrow}_{0}(\Z_{\ge
  0};\R_{>0},d)$, and every $n\in \Z_{\ge 0}$, we have
\begin{equation}\label{eq:6}
p^{\omega}_{K,\mathbf{a}_n}(X(f))\le 4N(n+1)\max_{i}\{p^{\omega}_{K,\mathbf{b}_{n}}(X^i)\} p^{\omega}_{K,\mathbf{a}_{n+1}}(f).
\end{equation}
\end{theorem}
\begin{proof}
Let $(U,\phi=(x^1,x^2,\ldots,x^N))$ be a coordinate chart on $M$. We first prove that, for
every $f,g\in C^{\omega}(M)$, every multi-index $(r)$ and every $x\in U$, we have
\begin{multline*}
\left\|D^{(r)}(fg)(x)\right\|\le\\ \sum_{j=0}^{|r|} \binom{|r|}{j}\sup\left\{\left\|(D^{(l)}f(x))\right\|\mid |l|=j\right\}\sup\left\{\left\|(D^{(l)}g(x))\right\|\mid |l|=|r|-j\right\}.
\end{multline*}
We prove this by induction on $|r|$. If $|r|=1$, then it is clear
that, for every $x\in U$, we have
\begin{equation*}
\left\|\frac{\partial}{\partial
    x^i}(fg)(x)\right\|=\left\|\frac{\partial f}{\partial
    x^i}(x)g(x)+\frac{\partial g}{\partial x^i}(x)f(x)\right\| \le \left\|\frac{\partial f}{\partial
    x^i}(x)g(x)\right\|+\left\|\frac{\partial g}{\partial x^i}(x)f(x)\right\|.
\end{equation*}
Now suppose that, for every $x\in U$ and for every $(r)$ such that $|r|\in \{1,2,\ldots,k\}$, we have 
\begin{multline*}
\left\|D^{(r)}(fg)(x)\right\|\le\\ \sum_{j=0}^{|r|}
\binom{|r|}{j}\sup\left\{\left\|(D^{(l)}f(x))\right\|\mid |l|=j\right\}\sup\left\{\left\|(D^{(l)}g(x))\right\|\mid |l|=|r|-j\right\}.
\end{multline*}
Let $(l)$ be a multi-index with $|l|=k+1$. Then there exists
$i\in\{1,2,\ldots,N\}$ and $(r)$ with $|r|=k$ such that
$(l)=(r)+(\widehat{i})$. So, for every $x\in U$, we have
\begin{multline*}
\left\|D^{(l)}(fg)(x)\right\|=\left\|D^{(r)}\left(\frac{\partial}{\partial
    x^i}(fg) \right) (x)\right\|\\\le \left\|D^{(r)}\left(\frac{\partial f}{\partial
    x^i} g \right) (x)\right\|+\left\|D^{(r)}\left(\frac{\partial g}{\partial
    x^i} f \right)(x)\right\|\\\le \sum_{j=0}^{|r|}
\binom{|r|}{j}\sup\left\{\left\|(D^{(l)}\frac{\partial f}{\partial
    x^i}(x))\right\|\mid |l|=j\right\}\sup\left\{\left\|(D^{(l)}g(x))\right\|\mid
|l|=|r|-j\right\}+\\\binom{|r|}{j}\sup\left\{\left\|(D^{(l)}f(x))\right\|\mid
|l|=j\right\}\sup\left\{\left\|(D^{(l)}\frac{\partial g}{\partial
    x^i}(x))\right\|\mid |l|=|r|-j\right\}\\=\sum_{j=0}^{|r|} \left(\binom{|r|}{j-1}+\binom{|r|}{j}\right) \sup\left\{\left\|(D^{(l)}f(x))\right\|\mid
|l|=j\right\}\times\\\sup\left\{\left\|(D^{(l)}g(x))\right\|\mid
|l|=|r|-j+1\right\}\\=\sum_{j=0}^{|r|}
\binom{|r|+1}{j}\sup\left\{\left\|(D^{(l)}f(x))\right\|\mid |l|=j\right\}\sup\left\{\left\|(D^{(l)}g(x))\right\|\mid |l|=|r|-j+1\right\}.
\end{multline*}
This completes the induction. Note that in the coordinate
neighbourhood $U$, we have
\begin{equation*}
X(f)=\sum_{i=1}^{N}X(x^i)\frac{\partial f}{\partial x^i}.
\end{equation*}
Thus, for every $x\in U$, we get
\begin{multline}\label{eq:4}
\left\|D^{(r)}(X(f))(x)\right\|\le\\ \sum_{j=0}^{|r|}\sum_{i=1}^{N}
\binom{|r|}{j}\sup\left\{\left\|(D^{(l)}X^i(x))\right\|\mid |l|=|r|-j\right\}\sup\left\{\left\|D^{(l)}\frac{\partial
    f}{\partial x^i}(x)\right\|\mid |l|=j\right\}.
\end{multline}
Now let $d>0$ and $\mathbf{a}\in \mathbf{c}^{\downarrow}_{0}(\Z_{\ge 0},
\R_{>0},d)$. Multiplying both sides of equation \eqref{eq:4} by
$\frac{(a_{n,0})(a_{n,1})\ldots (a_{n,|r|})}{|r|!}$, we get
\begin{multline*}
\frac{(a_{n,0})(a_{n,1})\ldots (a_{n,|r|})}{|r|!}\left\|D^{(r)}(X(f))(x)\right\|\le\\  
\sum_{i=1}^{N}\sum_{l=0}^{|r|} \left(\frac{(a_{n,0})(a_{n,1})\ldots
    (a_{n,l+1})}{l!}\sup\left\{\left\|D^{(s)}\frac{\partial f}{\partial
      x^i}(x)\right\|\mid |s|=l\right\}\right) \\\times\left(\frac{(a_{n,l+2})(a_{n,l+3})\ldots
    (a_{n,|r|})}{(|r|-l)!}\sup\left\{\left\|D^{(s)}X^i(x)\right\|\mid |s|=|r|-l\right\} \right),\qquad\forall x\in U.
\end{multline*}
Since the sequence $\mathbf{a}_n$ is decreasing, we have
\begin{multline*}
\frac{(a_{n,0})(a_{n,1})\ldots (a_{n,|r|})}{|r|!}\left\|D^{(r)}(X(f))(x)\right\|\\\le
\sum_{i=1}^{N}\sum_{l=0}^{|r|} \left(\frac{(a_{n,0})(a_{n,1})\ldots
    (a_{n,l+1})}{l!}\sup\left\{\left\|D^{(s)}\frac{\partial f}{\partial
      x^i}(x)\right\|\mid |s|=l\right\}\right)\\\times \left(\frac{(a_{n,0})(a_{n,1})\ldots
    (a_{n,|r|-l-2})}{(|r|-l)!}\sup\left\{\left\|D^{(s)}X^i(x)\right\|\mid |s|=|r|-l\right\} \right),\qquad\forall x\in U.
\end{multline*}
Using the above lemma, we have
\begin{eqnarray*}
\frac{(a_{n,0})(a_{n,1})\ldots (a_{n,l+1})}{(l)!}&\le&
                                                   (n+1)\frac{(a_{n+1,0})(a_{n+1,1}\ldots (a_{n+1,l+1})}{(l+1)!},\\
\frac{(a_{n,0})(a_{n,1})\ldots (a_{n,|r|-l-2})}{(|r|-l-2)!}&=&
                                                   \frac{(b_{n+1,0})(b_{n+1,1})\ldots (b_{n+1,|r|-l})}{(|r|-l)!}
\end{eqnarray*}
Therefore, we get
\begin{multline*}
\frac{(a_{n,0})(a_{n,1})\ldots (a_{n,|r|})}{|r|!}\left\|D^{(r)}(X(f))(x)\right\|\\\le
\sum_{i=1}^{N}\sum_{l=0}^{|r|} \frac{(n+1)}{(|r|-l)(|r|-l-1)}\left(\frac{(a_{n+1,0})(a_{n+1,1})\ldots
    (a_{n+1,l+1})}{(l+1)!}\sup\left\{\left\|D^{(s)}f(x)\right\|\mid |s|=l+1\right\}\right)\\ \left(\frac{(b_{n,0})(b_{n,1})\ldots
    (b_{n,|r|-l})}{(|r|-l)!}\sup\left\{\left\|D^{(s)}X^i(x)\right\|\mid |s|=|r|-l\right\} \right),\quad\forall x\in U.
\end{multline*}
Thus, by taking the supremum over $l\in \Z_{\ge 0}$ and $x\in K$ of the two term
in the right hand side of the above inequality, we get
\begin{multline*}
\frac{(a_{n,0})(a_{n,1})\ldots (a_{n,|r|})}{|r|!}\left\|D^{(r)}(X(f))(x)\right\|\\\le
N(n+1)p^{\omega}_{K,\mathbf{a}_{n+1}}(f)
p^{\omega}_{K,\mathbf{b}_n}(X^i)\sum_{l=0}^{|r|}
\frac{1}{(|r|-l)(|r|-l-1)}\\\le 4N(n+1)p^{\omega}_{K,\mathbf{a}_{n+1}}(f)
p^{\omega}_{K,\mathbf{b}_n}(X^i),\qquad\forall x\in U.
\end{multline*}
By taking the supremum of the left hand
side of the above inequality over $|r|\in \N$ and $x\in K$, for every $\mathbf{a}\in
\mathbf{c}^{\downarrow}_{0}(\Z_{\ge 0}; \R_{>0},d)$, we get
\begin{equation*}
p^{\omega}_{K,\mathbf{a}_n}(X(f))\le 4N(n+1)\max_{i}\{p^{\omega}_{K,\mathbf{b}_{n}}(X^i)\} p^{\omega}_{K,\mathbf{a}_{n+1}}(f).
\end{equation*}
\end{proof}

Using the $C^{\omega}$-topology on the space $\Gamma^{\omega}(TM)$,
one can study different properties of time-varying real analytic vector fields as curves on
$\Gamma^{\omega}(TM)$. In this part, we introduce the notions of
integrability and absolute continuity for curves on locally convex
spaces.

\begin{definition}
Let $V$ be a locally convex space with a family of generating
seminorms $\{p_i\}_{i\in \Lambda}$ and let $\mathbb{T}\subseteq \R$ be an
interval. 
A curve $f:\mathbb{T}\to V$ is \textbf{integrally bounded} if, for every
$i\in \N$, we have
\begin{equation*}
\int_{\mathbb{T}} p_i(f(\tau))d\tau<\infty.
\end{equation*}

A function $s:\mathbb{T}\to V$ is a \textbf{simple function} if there
exist $n\in \N$, measurable sets $A_1,A_2\ldots,A_n\subseteq
\mathbb{T}$, and $v_1,v_2,\ldots,v_n\in V$ such that
$\mathfrak{m}(A_i)<\infty$ for every $i\in\{1,2,\ldots,n\}$ and
\begin{equation*}
s=\sum_{i=1}^{n} \chi_{A_i}v_i.
\end{equation*} 

The set of all simple functions from the interval $\mathbb{T}$ to the
vector space $V$ is denoted by $S(\mathbb{T};V)$. 

One can define \textbf{Bochner integral} of a simple function
$s=\sum_{i=1}^{n} \chi_{A_i}v_i$ as
\begin{equation*}
\int_{\mathbb{T}} s(\tau)d\tau=\sum_{i=1}^{n} \mathfrak{m}(A_i) v_i.
\end{equation*}
It is easy to show that the above expression does not depend on choice
of $A_1,A_2,\ldots, A_n\subseteq \mathbb{T}$.

A curve $f:\mathbb{T}\to V$ is \textbf{Bochner
  approximable} if there exists a net $\{f_{\alpha}\}_{\alpha\in
  \Lambda}$ of simple functions on $V$ such that, for every seminorm
$p_i$, we have
\begin{equation*}
\lim_{\alpha} \int_{\mathbb{T}} p_i(f_{\alpha}(\tau)-f(\tau))d\tau=0.
\end{equation*}
The net of simple functions $\{f_{\alpha}\}_{\alpha\in\Lambda}$ is
an \textbf{approximating net} for the mapping $f$.
\end{definition}
\begin{theorem}[\cite{RB-AD:2011}]
Let $\{f_{\alpha}\}_{\alpha\in\Lambda}$ be an approximating net for
the mapping $f:\mathbb{T}\to V$. Then
$\{\int_{\mathbb{T}}f_{\alpha}(\tau)d\tau\}_{\alpha\in\Lambda}$ is a Cauchy net. 
\end{theorem}
Let $f:\mathbb{T}\to V$ be a mapping and let
$\{f_{\alpha}\}_{\alpha\in\Lambda}$ be an approximating net of simple functions for $f$. If the net
$\{\int_{\mathbb{T}}f_{\alpha}(\tau)d\tau\}_{\alpha\in\Lambda}$
converges, then we say that $f$ is \textbf{Bochner integrable}. One
can show that the limit of
$\{\int_{\mathbb{T}}f_{\alpha}(\tau)d\tau\}_{\alpha\in\Lambda}$
doesn't depend on the choice of 
approximating net and is called \textbf{Bochner integral of $f$}. The set of all Bochner integrable curves from $\mathbb{T}$ to $V$ is
denoted by $\mathrm{L}^1(\mathbb{T};V)$. 

A curve $f:\mathbb{T}\to V$ is \textbf{locally Bochner integrable}
if for every compact set $J\subseteq \mathbb{T}$, the map $f\mid_J$ is
Bochner integrable. The set of all locally Bochner integrable curves from $\mathbb{T}$ to $V$ is
denoted by $\mathrm{L}^1_{\loc}(\mathbb{T}; V)$.

\begin{theorem}\label{th:1000}
Let $V$ be a complete, separable locally convex space, 
$\mathbb{T}\subseteq \R$ be an interval, and $f:\mathbb{T}\to V$
be a curve on $V$. Then $f$ is locally integrally bounded if and only if it is locally
Bochner integrable.
\end{theorem}

Using the $C^{\nu}$-topology on the space $\Gamma^{\nu}(TM)$,  one can
apply the Theorem \ref{th:60} and Theorem \ref{th:1000} to get the following
result. 

\begin{theorem}\label{th:3}
Let $X:\mathbb{T}\to \Gamma^{\nu}(TM)$
be a time-varying $C^{\nu}$-vector fields. Then $X$ is locally integrally bounded if and only if it is locally
Bochner integrable.
\end{theorem}

We denoted the space of Bochner integrable curves from a compact interval
$\mathbb{T}\subseteq\R$ to a locally
convex vector space $V$ by $\mathrm{L}^1(\mathbb{T};V)$. One can show
that $\mathrm{L}^1(\mathbb{T};V)$ is a vector space. Let
$\{p_i\}_{i\in \Lambda}$ be a family of generating seminorms for
$V$. Then, for every $i\in \Lambda$, one can define a seminorm
$p_{i,\mathbb{T}}$ on $\mathrm{L}^1(\mathbb{T};V)$ by
\begin{equation*}
p_{i,\mathbb{T}}(f)=\int_{\mathbb{T}} p_i\left(f(\tau)\right)d\tau.
\end{equation*}
Therefore, one can consider $\mathrm{L}^1(\mathbb{T};V)$ as a locally
convex space with the generating family of seminorms
$\{p_{i,\mathbb{T}}\}_{i\in \Lambda}$.

It would be
interesting to investigate whether this locally convex space can be characterized
using the locally convex space
space $V$ and the Banach space $\mathrm{L}^1(\mathbb{T})$. 

\begin{theorem}[\cite{HJ:1981}]\label{th:12}
Let $\mathbb{T}\subseteq\R$ and $V$ be a complete locally convex space. Then
there exists a linear homeomorphism between
$\mathrm{L}^1(\mathbb{T};V)$ and $\mathrm{L}^1(\mathbb{T})\widehat{\otimes}_{\pi} V$.
\end{theorem}

One can find the similar characterizations for the space of continuous
mappings from $\mathbb{T}$ to the locally convex space $V$.

\begin{theorem}[\cite{HJ:1981}]\label{th:14}
Let $\mathbb{T}\subseteq\R$ be a compact interval and $V$ be a complete locally convex space. Then
there exists a linear homeomorphism between
$\mathrm{C}^0(\mathbb{T};V)$ and $\mathrm{C}^0(\mathbb{T})\widehat{\otimes}_{\epsilon} V$.
\end{theorem}

It is possible to define different notions
of absolute continuity for a curve on a locally convex space $V$. In
this paper, we choose to use the following notion which turns out to be the
most applicable one in our study of flows of time-varying vector fields.

\begin{definition}\label{def:2}
A curve $f:\mathbb{T}\to V$ is \textbf{absolutely continuous} if there
exists a Bochner integrable curve $g:\mathbb{T}\to V$ such that, for every
$t_0\in \mathbb{T}$, we have
\begin{equation*}
f(t)=f(t_0)+\int_{t_0}^{t}g(\tau)d\tau,\qquad\forall t\in\mathbb{T}.
\end{equation*}
The set of all absolutely continuous curves on $V$ on the
interval $\mathbb{T}$ is denoted by $\mathrm{AC}(\mathbb{T};V)$.
\end{definition}

\begin{theorem}
Let $\xi:\mathbb{T}\to \LC{M}{N}{\nu}$ be a locally absolutely continuous
curve on $\LC{M}{N}{\nu}$. Then $\xi$ is 
differentiable for almost every $t\in\mathbb{T}$.
\end{theorem}
\begin{proof}
Without loss of generality, we assume that $\mathbb{T}$ is compact. Then there exists $\eta\in
\mathrm{L}^1(\mathbb{T};\LC{M}{N}{\nu})$ such that
\begin{equation*}
\xi(t)=\xi(t_0)+\int_{t_0}^{t}\eta(\tau)d\tau,\qquad\forall t\in\mathbb{T}.
\end{equation*}
Therefore, it suffice to show that, for almost every $t_0\in \mathbb{T}$, we have 
\begin{equation*}
\limsup_{t\to t_0} \frac{1}{t-t_0} \int_{t_0}^{t}\left(\eta(\tau)-\eta(t_0)\right)d\tau=0.
\end{equation*}
Since $C^0(\mathbb{T})$ is dense in $\mathrm{L}^1(\mathbb{T})$, the
set $C^0(\mathbb{T})\widehat{\otimes}_{\pi} \LC{M}{N}{\nu}$ is dense
in $\mathrm{L}^1(\mathbb{T}) \widehat{\otimes}_{\pi}
\LC{M}{N}{\nu}$ \cite[\S 15.2, Proposition 3(a)]{HJ:1981}. Since the locally convex space $\LC{M}{N}{\nu}$ is complete, by
Theorem \ref{th:12} and Theorem \ref{th:14}, we have
\begin{eqnarray*}
C^0(\mathbb{T})\widehat{\otimes}_{\pi}
  \LC{M}{N}{\nu})&=&C^0(\mathbb{T};\LC{M}{N}{\nu}),\\
\mathrm{L}^1(\mathbb{T}) \widehat{\otimes}_{\pi}
\LC{M}{N}{\nu})&=&\mathrm{L}^1(\mathbb{T};\LC{M}{N}{\nu}).
\end{eqnarray*}
This implies that $C^0(\mathbb{T};\LC{M}{N}{\nu})$ is dense
in $\mathrm{L}^1(\mathbb{T};\LC{M}{N}{\nu})$. Let $\{p_i\}_{i\in I}$
be a generating family of seminorms for $\LC{M}{N}{\nu}$. For
$\epsilon>0$ and $i\in I$, there exists $g\in C^0(\mathbb{T};\LC{M}{N}{\nu})$ such that 
\begin{equation*}
\int_{\mathbb{T}}p_{i}(g(\tau)-\eta(\tau))d\tau<\epsilon.
\end{equation*}
So we assume that $t>t_0$ and we can write
\begin{multline}\label{eq:8}
\frac{1}{t-t_0}\int_{t_0}^{t}p_i\left(\eta(\tau)-\eta(t_0)\right)d\tau
\le \frac{1}{t-t_0}\int_{t_0}^{t} p_i(\eta(\tau)-g(\tau))d\tau\\+\frac{1}{t-t_0}\int_{t_0}^{t} p_i\left(g(\tau)-g(t_0)\right)d\tau+p_{i}(g(t_0)-\eta(t_0)).
\end{multline}
Since $g$ is continuous, we get
\begin{equation*}
\limsup_{t\to t_0}\frac{1}{t-t_0}\int_{t_0}^{t} p_i\left(g(\tau)-g(t_0)\right)d\tau=0.
\end{equation*}
If we take limit supremum of both side of \eqref{eq:8}, we
have
\begin{multline*}
\limsup_{t\to t_0}\left(\frac{1}{t-t_0}\int_{t_0}^{t}p_i\left(\eta(\tau)-\eta(t_0)\right)d\tau \right)\\\le
\limsup_{t\to t_0}\left(\frac{1}{t-t_0}\int_{t_0}^{t} p_i(\eta(\tau)-g(\tau))d\tau\right)+p_{i}(g(t_0)-\eta(t_0)).
\end{multline*}
Now suppose that there exists a set $A$ such that $\mathfrak{m}(A)\ne 0$
and we have
\begin{equation*}
\limsup_{t\to t_0} \left(\frac{1}{t-t_0}
\int_{t_0}^{t}p_i\left(\eta(\tau)-\eta(t_0)\right)d\tau\right)\ne
0,\qquad\forall t_0\in A.
\end{equation*}
This implies that, there exists $\alpha>0$ such that the set $B$ defined as 
\begin{equation*}
B=\left\{t_0\in \mathbb{T}\mid 
\limsup_{t\to t_0}\left(\frac{1}{t-t_0}\int_{t_0}^{t}p_i\left(\eta(\tau)-\eta(t_0)\right)d\tau\right)>\alpha\right\}.
\end{equation*}
has positive Lebesgue measure. However, we have 
\begin{equation*}
\int_{\mathbb{T}} p_i\left(g(\tau)-\eta(\tau)\right)d\tau=
\int_C p_i\left(g(\tau)-\eta(\tau)\right)d\tau+
\int_D p_i\left(g(\tau)-\eta(\tau)\right)d\tau.
\end{equation*}
Where $C,D\subseteq \mathbb{T}$ are defined as 
\begin{eqnarray*}
C&=&\{t_0\in \mathbb{T}\mid p_{i}(g(t_0)-\eta(t_0))>\frac{\alpha}{2}\},\\
D&=&\{t_0\in \mathbb{T}\mid p_{i}(g(t_0)-\eta(t_0)) \le \frac{\alpha}{2}\}.
\end{eqnarray*} 
This implies that 
\begin{equation*}
\int_C p_i\left(g(\tau)-\eta(\tau)\right)d\tau\ge \mathfrak{m}\{C\}\frac{\alpha}{2}.
\end{equation*}
Therefore we have
\begin{equation*}
\int_{\mathbb{T}} p_i\left(g(\tau)-\eta(\tau)\right)d\tau\ge 
\int_C p_i\left(g(\tau)-\eta(\tau)\right)d\tau\ge \mathfrak{m}\{C\}\frac{\alpha}{2}.
\end{equation*}
This means that
\begin{equation*}
\mathfrak{m}\left\{t_0\in \mathbb{T}\mid
p_{i}(g(t_0)-\eta(t_0))>\frac{\alpha}{2}\right\}\le \frac{2}{\alpha}
\int_{\mathbb{T}} p_i\left(g(\tau)-\eta(\tau)\right)d\tau<\frac{2\epsilon}{\alpha}.
\end{equation*}
Also, by \cite[Chapter 1, Theorem 4.3(a)]{JBG:2007}, we have
\begin{multline*}
\mathfrak{m}\left\{t_0\in \mathbb{T}\mid \limsup_{t\to t_0}\left(\frac{1}{t-t_0}\int_{t_0}^{t}
p_i(\eta(\tau)-g(\tau))d\tau\right)>\frac{\alpha}{2}\right\}\\\le \frac{4}{\alpha}
\int_{\mathbb{T}} p_i\left(g(\tau)-\xi(\tau)d\tau\right)<\frac{4\epsilon}{\alpha}.
\end{multline*}
So this implies that 
\begin{multline*}
\mathfrak{m}(B)\le \mathfrak{m}\left\{t_0\in \mathbb{T}\mid
p_{i}(g(t_0)-\eta(t_0))>\frac{\alpha}{2}\right\}\\+\mathfrak{m}\left\{t_0\in
\mathbb{T}\mid \limsup_{t\to t_0}\left(\frac{1}{t-t_0}\int_{t_0}^{t}
p_i(\eta(\tau)-g(\tau))d\tau\right)>\frac{\alpha}{2}\right\}\le\frac{6\epsilon}{\alpha}.
\end{multline*}
Since $\epsilon$ can be chosen arbitrary small, this is a
contradiction. 
\end{proof}

It is easy to see that the space $\ACC{T}{M}{N}{\omega}$ is a vector
space. Let $\{p^{\omega}_{K,\mathbf{a},f}\}$ be the family of generating
seminorms for the $C^{\omega}$-topology on $\LC{M}{N}{\omega}$ and let $\mathbb{T}\subseteq \R$ be an interval. For every compact
subinterval $\mathbb{I}\subseteq\mathbb{T}$, we define the seminorm
$q^{\omega}_{K,\mathbf{a},f,\mathbb{I}}$ as
\begin{equation*}
q^{\omega}_{K,\mathbf{a},f,\mathbb{I}}(X)=\int_{\mathbb{I}} p^{\omega}_{K,\mathbf{a},f}\left(\frac{d
  X}{d\tau}(\tau)(f)\right)d\tau.
\end{equation*}
The family of seminorms $\{p^{\omega}_{K,\mathbf{a},f,\mathbb{I}},
q^{\omega}_{K,\mathbf{a},f, \mathbb{I}}\}$ generates a locally convex topology on the
space $\ACC{T}{M}{N}{\omega}$.

\section{Global extension of real analytic vector fields}\label{sec:8}

As mentioned in the introduction, not every time-varying real analytic
vector field can be extended to a holomorphic one on a 
neighbourhood of its
domain. However, by imposing some appropriate joint condition on time and
state, one can show that such an extension exists. In this
section, we show that every ``locally integrally bounded'' time-varying real
analytic vector field on a real analytic manifold $M$, can be extended
to a locally Bochner integrable, time-varying
holomorphic vector field on a complex neighbourhood of $M$. Moreover,
we show that if $X$ is a continuous time-varying real
analytic vector field, then its extension $\overline{X}$ is a continuous
time-varying holomorphic vector field. 

We state the following lemma which turns out to be useful in
studying extension of real analytic vector fields.  The proof of the
first lemma is given in \cite[Corollary 1]{OH:1963}.

\begin{lemma}\label{lem:1}
Let $\Lambda$ be a directed set and $(E_{\alpha},\{i_{\alpha\beta}\})_{\beta\succeq\alpha}$ be an
inductive family of locally convex
spaces with locally convex inductive limit $(E,\{i_{\alpha}\}_{\alpha\in\Lambda})$. Let $F$ be a subspace
of $E$ such that, for every $\alpha\in \Lambda$, we have
\begin{equation*}
E_{\alpha}=\mathrm{cl}_{E_{\alpha}}\left(i_{\alpha}^{-1}(F)\right).
\end{equation*}
Then $F$ is a dense subset of $E$.
\end{lemma}
Having a directed set $\Lambda$ and an inductive family of locally convex spaces
$(E_{\alpha},\{i_{\alpha\beta}\})_{\beta\succeq\alpha }$, for every
$\beta\succeq \alpha$, one can
define $\tilde{i}_{\alpha\beta}:\mathrm{L}^1(\mathbb{T};E_{\alpha})\to
\mathrm{L}^1(\mathbb{T};E_{\beta})$ as
\begin{equation*}
\tilde{i}_{\alpha\beta}(f)(t)=i_{\alpha\beta}(f(t)),\qquad\forall t\in \mathbb{T}.
\end{equation*}
We can also define the map
$\tilde{i}_{\alpha}:\mathrm{L}^1(\mathbb{T};E_{\alpha})\to\mathrm{L}^1(\mathbb{T};E)$ as
\begin{equation*}
\tilde{i}_{\alpha}(f)(t)=i_{\alpha}(f(t)).
\end{equation*}
Then it is clear that
$(\mathrm{L^1}(\mathbb{T};E_{\alpha}),\{\tilde{i}_{\alpha\beta}\})_{\beta\succeq\alpha
}$ is an inductive family of locally convex spaces. 

\begin{lemma}\label{lem:2}
Let $\mathbb{T}\subseteq \R$ be a compact interval, $\Lambda$ be a directed set, and $(E_{\alpha},\{i_{\alpha\beta}\})_{\beta,\alpha\in\Lambda}$ be an
inductive family of locally convex
spaces with locally convex inductive limit
$(E,\{i_{\alpha}\}_{\alpha\in\Lambda})$. Then
$(\mathrm{L}^1(\mathbb{T}; E_{\alpha}),\{\tilde{i}_{\alpha\beta}\})_{\beta,\alpha\in\Lambda}$ is an
inductive family of locally convex
spaces with locally convex inductive limit
$(\mathrm{L}^1(\mathbb{T}; E),\{\tilde{i}_{\alpha}\}_{\alpha\in\Lambda})$.
\end{lemma}
\begin{proof}
Since $\mathrm{L}^1(\mathbb{T})$ is a normable space, by
\cite[Corollary 4, \S 15.5]{HJ:1981}, we have $\varinjlim_{\alpha} \mathrm{L}^1(\mathbb{T})\otimes_{\pi}
E_{\alpha}=\mathrm{L}^1(\mathbb{T})\otimes_{\pi}
E$. Let $F=\mathrm{L}^1(\mathbb{T})\otimes_{\pi} E$. Then, for every
$\alpha\in \Lambda$, we have 
\begin{equation*}
 \mathrm{L}^1(\mathbb{T})\otimes_{\pi}E_{\alpha}\subseteq
 \tilde{i}_{\alpha}^{-1}(F).
\end{equation*}
This implies that
\begin{equation*}
\mathrm{L}^1(\mathbb{T}; E_{\alpha}) = \mathrm{cl}\left( \tilde{i}_{\alpha}^{-1}(F)\right).
\end{equation*}
Then by using Lemma \ref{lem:1}, we have that $F$ is a dense subset
of $\varinjlim_{\alpha} \mathrm{L}^1(\mathbb{T}; E_{\alpha})$. This means that $\varinjlim_{\alpha} \mathrm{L}^1(\mathbb{T}; E_{\alpha})=\mathrm{L}^1(\mathbb{T}; E)$.
\end{proof}

Using Lemmata \ref{lem:1} and \ref{lem:2}, one can deduce the
following result which we refer to as the global extension of
real analytic vector fields.

\begin{theorem}\label{th:21}
Let $M$ be a real analytic manifold and let $\mathscr{N}_M$ be the
family of all neighbourhoods of $M$. Then we have
\begin{equation*}
\varinjlim_{\overline{U}_M\in \mathscr{N}_M} \mathrm{L}^1(\mathbb{T};\Gamma^{\hol,\R}(\overline{U}_M))=\mathrm{L}^1(\mathbb{T};\Gamma^{\omega}(TM)).
\end{equation*}
\end{theorem}

\begin{corollary}\label{th:11}
Let $X\in \mathrm{L}^1(\mathbb{T};\Gamma^{\omega}(TM))$. There exists a
neighbourhood $\overline{U}_M$ of $M$ and a locally Bochner
integrable time-varying holomorphic vector field $\overline{X}\in
\mathrm{L}^1(\mathbb{T};\Gamma^{\hol}(\overline{U}_M))$ such that
$\overline{X}(t,x)=X(t,x)$, for every $t\in \mathbb{T}$ and every
$x\in M$.
\end{corollary}

Similarly, one can study the extension of continuous time-varying real analytic
vector fields. While a continuous time-varying real analytic
vector fields is locally Bochner integrable, it has a
holomorphic extension to a suitable domain. However, this raises the
question of whether
the holomorphic extension of a ``continuous'' time-varying real
analytic vector field is a
``continuous'' time-varying holomorphic vector field or not. Using the
following lemma, we show that the answer to the above question is positive.

\begin{lemma}
Let $K$ be a compact topological space, $\Lambda$ be a directed
set, and $(E_{\alpha},\{i_{\alpha\beta}\})_{\beta\succeq\alpha }$ be an
inductive family of nuclear locally convex spaces with locally convex inductive
limit $(E, \{i_{\alpha}\})_{\alpha\in\Lambda }$. Suppose that $E$ is
also a nuclear space. Then $(\mathrm{C}^0(K;E_{\alpha}),\{\hat{i}_{\alpha\beta}\})_{\beta\succeq\alpha
}$ is an inductive family of locally convex spaces with inductive limit
$(\mathrm{C}^0(K; E) ,\{\hat{i}_{\alpha}\}_{\alpha \in\Lambda})$.
\end{lemma}
\begin{proof}
Since $\mathrm{C}^0(K)$ is a normable space, by
\cite[Corollary 4, \S 15.5]{HJ:1981}, we have $\varinjlim_{\alpha} \mathrm{C}^0(K)\otimes_{\pi}
E_{\alpha}=\mathrm{C}^0(K)\otimes_{\pi}
E$. For every $\alpha\in \Lambda$,
the space $E_{\alpha}$ is nuclear. Therefore, by \cite[\S 21.3, Theorem 1]{HJ:1981}, we have
\begin{equation*}
\mathrm{C}^0(K)\otimes_{\pi} E_{\alpha}=\mathrm{C}^0(K)\otimes_{\epsilon}
E_{\alpha},\qquad\forall \alpha\in \Lambda.
\end{equation*}
Moreover, the space $E$ is nuclear. So, again using \cite[\S 21.3, Theorem 1]{HJ:1981}, we have
\begin{equation*}
\mathrm{C}^0(K)\otimes_{\pi}
E=\mathrm{C}^0(K)\otimes_{\epsilon} E.
\end{equation*}
This implies that 
\begin{equation*}
\varinjlim_{\alpha} \mathrm{C}^0(K)\otimes_{\epsilon}
E_{\alpha}=\mathrm{C}^0(K)\otimes_{\epsilon} E.
\end{equation*}
We set 
$F=\mathrm{C}^0(K)\otimes_{\epsilon}E$. Then, for every
$\alpha\in \Lambda$, we have 
\begin{equation*}
 \mathrm{C}^0(K)\otimes_{\epsilon}E_{\alpha}\subseteq
\hat{i}^{-1}_{\alpha}(F).
\end{equation*}
This implies that
\begin{equation*}
\mathrm{C}^0(K;E_{\alpha}) \subseteq \mathrm{cl}\left(\hat{i}^{-1}_{\alpha}F\right).
\end{equation*}
Then, by using Lemma \ref{lem:1}, we have that $F$ is a dense subset
of $\varinjlim_{\alpha} \mathrm{C}^0(K; E_{\alpha})$. This means that we have $\varinjlim_{\alpha}
\mathrm{C}^0(K; E_{\alpha})=\mathrm{C}^0(K; E)$. 
\end{proof}

\begin{theorem}\label{th:22}
Let $K$ be a compact topological space, $M$ be a real analytic vector field and $\mathscr{N}_M$ be the
family of all neighbourhoods of $M$, which is a directed set under
inclusion. Then we have
\begin{equation*}
\varinjlim_{\overline{U}_M\in\mathscr{N}_M} \mathrm{C}^0(K;
\Gamma^{\hol}(\overline{U}_M))=\mathrm{C}^0(K;\Gamma^{\omega}(TM)).
\end{equation*}
\end{theorem}
\begin{proof}
Let $\Lambda$ be a directed set and
$(E_{\alpha},\{i_{\alpha\beta}\})_{\beta\succeq\alpha }$ be a directed
system of locally convex spaces. Then, for every $\beta\succeq \alpha$, one can
define $\hat{i}_{\alpha\beta}:\mathrm{C}^0(K;E_{\alpha})\to
\mathrm{C}^0(K;E_{\beta})$ as
\begin{equation*}
\hat{i}_{\alpha\beta}(f)(u)=i_{\alpha\beta}(f(u)),\qquad\forall u\in K.
\end{equation*}
For every $\alpha\in \Lambda$, we can also define the map
$\hat{i}_{\alpha}:\mathrm{C}^0(K;E_{\alpha})\to\mathrm{C}^0(K;E)$ as
\begin{equation*}
\hat{i}_{\alpha}(f)(u)=i_{\alpha}(f(u)),\qquad\forall u\in K.
\end{equation*}
Then it is clear that
$(\mathrm{C}^0(K;E_{\alpha}),\{\hat{i}_{\alpha\beta}\})_{\beta\succeq\alpha
}$ is an inductive family of locally convex spaces. The result follows from the above lemma.
\end{proof}

\section{Local extension of real analytic vector fields} \label{sec:9}

In the previous section, we proved that every locally Bochner
integrable real analytic vector field on $M$ has a holomorphic
extension on a neighbourhood of $M$. However, this result is true for extending one vector field. It is natural to ask that, if we have
a family of locally integrally bounded real analytic vector fields on $M$, can
we extend every member of the family to holomorphic vector fields on
one 
neighbourhood of $M$? In order to answer this question, we need a finer
result for the extension of real analytic vector fields. We will see that the projective
limit representation of the space of real analytic vector fields helps us
to get this extension result.

\begin{theorem}\label{th:28}
Let $K\subseteq M$ be a compact set and 
$\{\overline{U}_n\}_{n\in\N}$ be a sequence of neighbourhoods of
$M$ such that
\begin{equation*}
\mathrm{cl}(\overline{U}_{n+1})\subseteq \overline{U}_n,\qquad\forall
n\in \N.
\end{equation*}
and $\bigcap_{n\in \N} \overline{U}_n=K$. Then we 
have $\varinjlim_{n\to\infty}
\mathrm{L}^1(\mathbb{T};\Gamma^{\hol}_{\bdd}(\overline{U}_n))=\mathrm{L}^1(\mathbb{T};\mathscr{G}^{\hol,\R}_K)$. Moreover
the direct limit is weakly compact and boundedly retractive. 
\end{theorem}
\begin{proof}
We know that, by Theorem \ref{th:17}, for every $n\in \N$, the map
$\rho^{\R}_{\overline{U}_n}:\Gamma^{\hol,\R}_{\bdd}(\overline{U}_n)\to
\Gamma^{\hol,\R}(\overline{U}_n)$ is a compact continuous map. Note that
every $n\in \N$, the map
$\mathrm{id}\otimes\rho^{\R}_{\overline{U}_n}:\mathrm{L}^1(\mathbb{T})\otimes_{\pi}\Gamma^{\hol,\R}_{\bdd}(\overline{U}_n)\to
\mathrm{L}^1(\mathbb{T})\otimes_{\pi}\Gamma^{\hol,\R}(\overline{U}_n)$
is defined by 
\begin{equation*}
\mathrm{id}\otimes\rho^{\R}_{\overline{U}_n}(\xi(t)\otimes
\eta)=\xi(t)\otimes \rho^{\R}_{\overline{U}_n}(\eta).
\end{equation*}
Since $\mathrm{L}^1(\mathbb{T})\otimes_{\pi}
\Gamma^{\hol,\R}_{\bdd}(\overline{U}_n)$ is a dense subset of
$\mathrm{L}^1(\mathbb{T} ; \Gamma^{\hol,\R}_{\bdd}(\overline{U}_n))$,
one can extend the map $\mathrm{id}\otimes\rho^{\R}_{\overline{U}_n}$ into
the map
$\mathrm{id}\widehat{\otimes}\rho^{\R}_{\overline{U}_n}:\mathrm{L}^1(\mathbb{T}
; \Gamma^{\hol,\R}_{\bdd}(\overline{U}_n))\to \mathrm{L}^1(\mathbb{T}
; \Gamma^{\hol,\R}(\overline{U}_n))$. We show that
$\mathrm{id}\widehat{\otimes}\rho^{\R}_{\overline{U}_n}$ is weakly compact. 

In order to show that $\mathrm{id}\widehat{\otimes}\rho^{\R}_{\overline{U}_n}$
is weakly compact, it suffices to show that for a bounded set $B\subset
\mathrm{L}^1(\mathbb{T};\Gamma^{\hol,\R}_{\bdd}(\overline{U}_n))$, the set
$\mathrm{id}\widehat{\otimes}\rho^{\R}_{\overline{U}_n}(B)$ is relatively
weakly compact in $\mathrm{L}^1(\mathbb{T};\Gamma^{\hol,\R}(\overline{U}_n))$. Since
$\mathrm{L}^1(\mathbb{T};\Gamma^{\hol,\R}(\overline{U}_n))$ is a
complete locally convex space, by Theorem \ref{th:52}, the set 
\begin{equation*}
\mathrm{cl}\left(\mathrm{id}\widehat{\otimes}\rho^{\R}_{\overline{U}_n}(B)\right)
\end{equation*}
is weakly compact if it is
weakly sequentially compact. Therefore, it suffices to show that
$\mathrm{cl}\left(\mathrm{id}\widehat{\otimes}\rho^{\R}_{\overline{U}_n}(B)\right)$ is
weakly sequentially compact. Let $\{f_n\}_{n=1}^{\infty}$ in
$\mathrm{cl}\left(\mathrm{id}\widehat{\otimes}\rho^{\R}_{\overline{U}_n}(B)\right)$.
Since
$\mathrm{cl}\left(\mathrm{id}\widehat{\otimes}\rho^{\R}_{\overline{U}_n}(B)\right)$
is bounded, for every seminorm $p$ on
$\Gamma^{\hol,\R}(\overline{U}_n)$, there exists $M>0$ such that
\begin{equation*}
p(\int_{\mathbb{T}} f_n(\tau)d\tau)\le \int_{\mathbb{T}}
p(f_n(\tau))d\tau\le M.
\end{equation*}
This implies that the sequence $\left\{\int_{\mathbb{T}}
  f_n(\tau)d\tau\right\}_{n=1}^{\infty}$ is bounded in
$\Gamma^{\hol,\R}(\overline{U}_n)$. Since
$\Gamma^{\hol,\R}(\overline{U}_n)$ is a nuclear locally convex space,
the sequence $\left\{\int_{\mathbb{T}}
  f_n(\tau)d\tau\right\}_{n=1}^{\infty}$ is relatively compact in
$\Gamma^{\hol,\R}(\overline{U}_n)$. Therefore, there is a subsequence $\{f_{n_r}\}_{r=1}^{\infty}$ of
$\{f_n\}_{n=1}^{\infty}$ such that
\begin{equation*}
\left\{\int_{\mathbb{T}}
  f_{n_r}(\tau)d\tau\right\}_{r=1}^{\infty}
\end{equation*}
is Cauchy in $\Gamma^{\hol,\R}(\overline{U}_n)$. 

Note that the strong dual of $\mathrm{L}^1(\mathbb{T})$ is
$\mathrm{L}^{\infty}(\mathbb{T})$ \cite[Chapter 8]{royden1988}. We also know that
$\Gamma^{\hol,\R}(\overline{U}_n$ is a nuclear complete metrizable
space and $\mathrm{L}^1(\mathbb{T})$ is a Banach space. Therefore, using \cite[Chapter IV,
Theorem 9.9]{schaefer}, the strong dual of
$\mathrm{L}^1(\mathbb{T};\Gamma^{\hol,\R}(\overline{U}_n)$ is exactly
$\mathrm{L}^{\infty}(\mathbb{T})\widehat{\otimes}_{\pi}
\left(\Gamma^{\hol,\R}(\overline{U}_n)\right)'_{\beta}$. We first show
that, for every $\xi\otimes\eta\in \mathrm{L}^{\infty}(\mathbb{T})\otimes
\left(\Gamma^{\hol,\R}(\overline{U}_n)\right)'$, the sequence 
\begin{equation*}
\left\{\xi\otimes\eta (f_{n_r}) \right\}_{r=1}^{\infty}
\end{equation*}
is Cauchy in $\R$. Note that we have
\begin{multline*}
\xi\otimes\eta (f_{n_r}-f_{n_s})=\int_{\mathbb{T}} \xi(t) \eta
(f_{n_s}(t)-f_{n_r}(t))dt \\\le M
\int_{\mathbb{T}}\eta(f_{n_s}(t)-f_{n_r}(t))dt=M\eta\left(\int_{\mathbb{T}}
  (f_{n_s}(t)-f_{n_r}(t))dt\right).
\end{multline*} 
Since the sequence $\left\{\int_{\mathbb{T}}
  f_{n_r}(\tau)d\tau\right\}_{r=1}^{\infty}$ is Cauchy in
$\Gamma^{\hol,\R}(\overline{U}_n)$, this implies that the sequence
$\left\{\xi\otimes\eta (f_{n_r}) \right\}_{r=1}^{\infty}$ is Cauchy in
$\R$. Now we show that, for every $\lambda\in \mathrm{L}^{\infty}(\mathbb{T})\widehat{\otimes}\left(\Gamma^{\hol,\R}(\overline{U}_n)\right)'$, the sequence 
\begin{equation*}
\left\{\lambda(f_{n_r}) \right\}_{r=1}^{\infty}
\end{equation*}
is Cauchy in $\R$. Note that $\mathrm{L}^{\infty}(\mathbb{T})\otimes_{\pi}
\left(\Gamma^{\hol,\R}(\overline{U}_n)\right)'$ is a dense
subset of $\mathrm{L}^{\infty}(\mathbb{T})\widehat{\otimes}_{\pi}
\left(\Gamma^{\hol,\R}(\overline{U}_n)\right)'_{\beta}$. So there
exist a net $\{\xi_{\alpha}\}_{\alpha\in \Lambda}$ in
$\mathrm{L}^{\infty}(\mathbb{T})$ and a net $\{\eta_{\alpha}\}_{\alpha\in \Lambda}$ in
$\left(\Gamma^{\hol,\R}(\overline{U}_n)\right)'$ such that
\begin{equation*}
\lim_{\alpha} \xi_{\alpha}\otimes \eta_{\alpha}=\lambda.
\end{equation*}
Thus, for every $\epsilon>0$, there exists $\theta\in \Lambda$ such that
\begin{equation*}
\left\|\xi_{\theta}\otimes \eta_{\theta}(v)-\lambda(v)\right\|\le
\frac{\epsilon}{3},\qquad\forall v\in \mathrm{cl}\left(\mathrm{id}\widehat{\otimes}\rho^{\R}_{\overline{U}_n}(B)\right).
\end{equation*}
Since the sequence $\left\{\xi_{\theta}\otimes\eta_{\theta} (f_{n_r})
\right\}_{r=1}^{\infty}$ is Cauchy in $\mathbb{F}$, for every
$\epsilon>0$, there exists
$\tilde{N}>0$ such that
\begin{equation*}
\left\|\xi_N\otimes
  \eta_N(f_{n_s}-f_{n_r})\right\|<\frac{\epsilon}{3},\qquad\forall r,s>\tilde{N}.
\end{equation*}
Thus, for every $\epsilon>0$, there exists $\tilde{N}>0$ such that
\begin{multline*}
\left\|\lambda(f_{n_s}-f_{n_r})\right\|\le
\left\|\lambda(f_{n_s}-f_{n_r})-\xi_{\theta}\otimes \eta_{\theta}
  (f_{n_s}-f_{n_r})\right\|+\left\|\xi_{\theta}\otimes\eta_{\theta}(f_{n_s}-f_{n_r})\right\|\\\le
\left\|\lambda(f_{n_s})-\xi_{\theta}\otimes \eta_{\theta}
  (f_{n_s})\right\|+\left\|\lambda(f_{n_r})-\xi_{\theta}\otimes \eta_{\theta}
  (f_{n_r})\right\|+\left\|\xi_{\theta}\otimes\eta_{\theta}(f_{n_s}-f_{n_r})\right\|<\epsilon.
\end{multline*}
Therefore, the sequence $\{f_{n_r}\}_{r=1}^{\infty}$ is weakly
Cauchy in $\mathrm{L}^1(\mathbb{T};
\Gamma^{\hol,\R}(\overline{U}_n))$. This completes the proof of weak
compactness of the map $\mathrm{id}\widehat{\otimes}
\rho^{\R}_{\overline{U}_n}:\mathrm{L}^1(\mathbb{T};
\Gamma^{\hol,\R}_{\bdd}(\overline{U}_n))\to \mathrm{L}^1(\mathbb{T};
\Gamma^{\hol,\R}(\overline{U}_n))$. Recall that in the proof of
Theorem \ref{th:2}, for every $n\in \N$, we defined
the continuous linear map $r^{\R}_n:\Gamma^{\hol,\R}(\overline{U}_n)\to
\Gamma^{\hol,\R}_{\bdd}(\overline{U}_{n+1})$ by 
\begin{equation*}
r^{\R}_n(X)=X\lvert_{\overline{U}_{n+1}}.
\end{equation*}
Then we have the following diagram:
\begin{equation*}
\xymatrix{ \Gamma^{\hol,\R}_{\bdd}({U}_{n}) \ar[r]^{\rho^{\R}_{{U}_n}}&
  \Gamma^{\hol,\R}({U}_{n})\ar[r]^{r^{\R}_n}&
  \Gamma^{\hol,\R}_{\bdd}({U}_{n+1}).}
\end{equation*}
Therefore, we have the following diagram:
\begin{equation*}
\xymatrix{ \mathrm{L}^1(\mathbb{T};\Gamma^{\hol,\R}_{\bdd}({U}_{n})) \ar[r]^{\mathrm{id}\widehat{\otimes}\rho^{\R}_{{U}_n}}&
  \mathrm{L}^1(\mathbb{T}; \Gamma^{\hol}({U}_{n}))\ar[r]^{\mathrm{id}\widehat{\otimes}r^{\R}_n}&
  \mathrm{L}^1(\mathbb{T}; \Gamma^{\hol}_{\bdd}({U}_{n+1})).}
\end{equation*}
Since, $\mathrm{id}\widehat{\otimes} \rho^{\R}_{\overline{U}_n}$ is
weakly compact, by \cite[\S 17.2, Proposition 1]{HJ:1981}, the composition
$\mathrm{id}\widehat{\otimes} \rho^{\R}_{\overline{U}_n}\scirc
\mathrm{id}\widehat{\otimes} r^{\R}_{\overline{U}_n}$ is weakly
compact. Therefore, the connecting maps in the inductive limit $\varinjlim_{n\to\infty}
\mathrm{L}^1(\mathbb{T};\Gamma^{\hol}_{\bdd}(\overline{U}_n))=\mathrm{L}^1(\mathbb{T};\mathscr{G}^{\hol,\R}_K)$
are weakly compact.

Using Theorem \ref{th:24}, if we can show that
the direct limit satisfies condition ($M$), then it would be boundedly
retractive. Since the inductive limit $\varinjlim_{n\to\infty}
\Gamma^{\hol}_{\bdd}(\overline{U}_n)=\mathscr{G}^{\hol,\R}_K$ is compact, by Theorem \ref{th:25}, it
satisfies condition ($M$). This means that there exists a sequence
$\{V_n\}_{n\in \N}$ such that, for every $n\in \N$, $V_n$ is an absolutely convex
neighbourhood of $0$ in $\Gamma^{\hol}_{\bdd}(\overline{U}_n)$ and there exists $M_n>0$ such that, for every $m>M_n$, the
topologies induced from $\Gamma^{\hol}_{\bdd}(\overline{U}_m)$ on
$V_n$ are all the same. Now consider
the sequence $\{\mathrm{L}^1(\mathbb{T};V_n)\}_{n\in \N}$. It is clear
that, for every $n\in \N$, $\mathrm{L}^1(\mathbb{T};V_n)$ is an
absolutely convex neighbourhood of $0$ in $\mathrm{L}^1(\mathbb{T};\Gamma^{\hol}_{\bdd}(\overline{U}_n))$. For
every seminorm $p$ on $\Gamma^{\hol}_{\bdd}(\overline{U}_n)$ and every $m>M_n$,
there exists a seminorm $q_m$ on $\Gamma^{\hol}_{\bdd}(\overline{U}_m)$ such that 
\begin{equation*}
p(v)\le q_m(v),\qquad\forall v\in V_n.
\end{equation*}
This implies that, for every $X\in \mathrm{L}^1(\mathbb{T};V_n)$, we have
\begin{equation*}
\int_{\mathbb{T}}p(X(\tau))d\tau \le \int_{\mathbb{T}}q_m(X(\tau))d\tau.
\end{equation*}
So, for every $m>M_n$, the topology induced on $\mathrm{L}^1(\mathbb{T};V_n)$ from
$\mathrm{L}^1(\mathbb{T};\Gamma^{\hol}_{\bdd}(\overline{U}_m))$ is the same as
its original topology. Therefore, the inductive limit  
\begin{equation*}
\varinjlim_{n\to\infty}
\mathrm{L}^1(\mathbb{T};\Gamma^{\hol}_{\bdd}(\overline{U}_n))=\mathrm{L}^1(\mathbb{T};\mathscr{G}^{\hol,\R}_K)
\end{equation*}
satisfies condition $(M)$ and it is boundedly retractive.
\end{proof}

Using the local extension theorem developed here, we can state
the following result, which can be considered as generalization of
Corollary \ref{th:11}. 

\begin{corollary}\label{th:13}
Let $B\subseteq \mathrm{L}^1(\mathbb{T};\Gamma^{\omega}(TM))$ be a bounded
set. Then, for every compact set $K\subseteq M$, there exists a
neighbourhood $\overline{U}_K$ of $K$ and a bounded set $\overline{B}\in
\mathrm{L}^1(\mathbb{T};\Gamma^{\hol}_{\bdd}(\overline{U}_n))$ such that, for
every $X\in B$, there exists a $\overline{X}\in \overline{B}$ such
that
\begin{equation*}
\overline{X}(t,x)=X(t,x)\qquad\forall t\in \mathbb{T}, \ \forall x\in K.
\end{equation*}
\end{corollary}

Let $M$ be a real analytic manifold and let $U\subseteq M$ be a relatively
compact subset of $M$. Then, by the local extension theorem, for every
$f\in C^{\omega}(M)$, there exists a neighbourhood
$\overline{V}\subseteq M^{\C}$ of $U$ such that $f$ can be extended to
a bounded holomorphic function
$\overline{f}\in C^{\hol}_{\bdd}(\overline{V})$. It is useful to study the relationship between
the seminorms of $f$ and the seminorms of its holomorphic extension
$\overline{f}$. 

\begin{theorem}\label{th:19}
Let $M$ be a real analytic manifold and $U$ be a relatively
compact subset of $M$. Then, for every neighbourhood $\overline{V}\subseteq
M^{\C}$ of $\mathrm{cl}(U)$, there
exists $d>0$ such that, for every $f\in C^{\omega}(M)$ with a
holomorphic extension $\overline{f}\in C_{\bdd}^{\hol}(\overline{V})$, we have
\begin{equation*}
p^{\omega}_{K,\mathbf{a}}(f)\le p_{\overline{V}}(\overline{f}),\qquad\forall \mathbf{a}\in
\mathbf{c}^{\downarrow}_{0}(\Z_{\ge 0},\R_{>0},d),\ \forall \mbox{ compact }K\subseteq U.
\end{equation*}
\end{theorem}
\begin{proof}
Since $\overline{f}$ is a holomorphic extension of $f$, we have
\begin{equation*}
\overline{f}(x)=f(x),\qquad\forall x\in \mathrm{cl}(U).
\end{equation*}
Since $\mathrm{cl}(U)$ is compact, one can choose $d>0$ such that, for
every $x\in \mathrm{cl}(U)$, we have $D_{(d)}(x)\subseteq
\overline{V}$, where $(d)=(d,d,\ldots,d)$. We set $D=\bigcup_{x\in U}
D_{(d)}(x)$. Then we have $D\subseteq \overline{V}$. 
Using Cauchy's estimate, for every multi-index $(r)$ and for every
$\mathbf{a}\in \mathbf{c}^{\downarrow}_{0}(\Z_{\ge 0},\R_{>0}, d)$, we have
\begin{equation*}
\frac{a_0a_1\ldots a_{|r|}}{(r)!}\|D^{(r)}f(x)\|\le
\frac{a_0}{d}\frac{a_1}{d}\ldots \frac{a_{|r|}}{d}\sup\left\{\|\overline{f}(x)\|\mid
x\in D\right\} \le p_{\overline{V}}(\overline{f}),\qquad\forall x\in U.
\end{equation*}
This implies that, for every compact set $K\subseteq U$ and every
$\mathbf{a}\in \mathbf{c}^{\downarrow}_{0}(\Z_{\ge 0},\R_{>0}, d)$, we have
\begin{equation*}
p^{\omega}_{K,\mathbf{a}}(f)\le p_{\overline{V}}(\overline{f}).
\end{equation*}
\end{proof}

\section{Series representation of flows of time-varying real analytic vector fields}\label{sec:10}

In this section, using the holomorphic extension theorems, we study
flows of time-varying real analytic vector fields. The operator framework that
we use for this analysis (as to our knowledge) has been first
introduced in \cite{agrachev1978exponential}.
As mentioned in the previous sections, a time-varying $C^{\omega}$-vector field can
be considered as a curve on the locally convex space
$\LC{M}{M}{\omega}$. Let $X:\mathbb{T}\times M\to TM$ be a
time-varying real analytic vector field. Then we define
$\widehat{X}:\mathbb{T}\to \mathrm{L}(C^{\omega}(M);C^{\omega}(M))$ as
\begin{equation*}
\widehat{X}(t)(f)=df(X(t)),\qquad\forall t\in \mathbb{T}, \ \forall
f\in C^{\omega}(M)
\end{equation*}
Following the analysis in \cite{agrachev1978exponential}, the flow of 
a time-varying $C^{\omega}$-vector field $X$ can be considered as a
curve $\zeta:\mathbb{T}\to \LC{M}{U}{\omega}$ which satisfies the
following initial value problem on the locally convex space $\LC{M}{U}{\omega}$:
\begin{eqnarray}\label{eq:1}
\begin{split}
\frac{d\zeta}{dt}(t)&=\zeta(t)\scirc \widehat{X}(t),\qquad a.e.\
                       t\in\mathbb{T}\\
\zeta(0)&=\mathrm{id}.
\end{split}
\end{eqnarray}
Therefore, one can reduce the problem of studying the flow of a time-varying vector
field to the problem of studying solutions of a linear differential equation on a
locally convex vector space.
The theory of ordinary differential equations on locally convex spaces is
different in nature from the classical theory of ordinary differential
equations on Banach spaces. In the
theory of differential equations on Banach spaces, there are many
general results about existence, uniqueness and properties of the
flows of vector fields, which hold independently of the underlying
Banach space. However, the
theory of ordinary differential equations on locally convex spaces
heavily depends on the nature of their underlying space. Many methods in
the classical theory of ordinary differential equations in Banach spaces have no counterpart
in the theory of ordinary differential equations on locally convex
spaces \cite{SGL-GOS:1994}. For instance, one can easily find
counterexamples for Peano's existence theorem for linear differential
equations on locally convex spaces \cite{SGL-GOS:1994}.

In \cite{agrachev1978exponential}, the initial value problem
\eqref{eq:1} for both time-varying smooth vector fields
and time-varying real analytic vector fields
has been studied on $\mathrm{L}(C^{\infty}(\R^n);C^{\infty}(\R^n))$. In
the real analytic case, $X$ is assumed to be a locally
integrally bounded time-varying $C^{\omega}$-vector field on $\R^n$ such that it
can be extended to a bounded holomorphic vector field on a
neighbourhood  $\Omega\subseteq \C^n$ of $\R^n$. Using the $C^{\hol}$-topology on the
space of holomorphic vector fields, it has
been shown that the well-known sequence of Picard iterations for the
initial value problem \eqref{eq:1} converges and gives us the unique solution of
\eqref{eq:1} \cite[\S 2, Proposition 2.1]{agrachev1978exponential}. In
the smooth case, the existence and uniqueness of solutions of \eqref{eq:1} has
been shown. However, for smooth but
not real analytic vector fields, the sequence of Picard iterations
associated to the initial value problem \eqref{eq:1} does not converge
\cite[\S 2.4.4]{agrachevcontrol2004}. 

In this section, we study the initial value problem
\eqref{eq:1} for the real analytic cases on the locally convex space $\mathrm{L}(C^{\omega}(M);C^{\omega}(M))$. Using the local
extension theorem \eqref{th:11} and estimates for seminorms on the space of real analytic
functions, we provide a direct method for proving and studying the
convergence of sequence of Picard iterations. This method helps
us to generalize the result of \cite[\S 2, Proposition 2.1]{agrachev1978exponential} to
arbitrary locally integrally bounded time-varying real analytic vector fields. 

\begin{theorem}\label{th:37}
Let $X:\mathbb{T}\to \Gamma^{\omega}(TM)$ be
a locally integrally bounded time-varying vector field. Then, for
every $t_0\in\mathbb{T}$ and every $x_0\in M$, there exists an interval
$\mathbb{T}'\subseteq\mathbb{T}$ containing $t_0$ and an open set
$U\subseteq M$ containing $x_0$ such that there exists a unique
locally absolutely continuous curve $\zeta:\mathbb{T}'\to
\mathrm{L}(C^{\omega}(M);C^{\omega}(U))$ which satisfies the following
initial value problem:
\begin{eqnarray}\label{eq:2}
\begin{split}
\frac{d\zeta}{dt}(t)&=\zeta(t)\scirc \widehat{X}(t),\qquad a.e.\
                       t\in\mathbb{T}',\\
\zeta(t_0)&=\mathrm{id},
\end{split}
\end{eqnarray}
and, for every $t\in\mathbb{T}'$, we have
\begin{equation}\label{eq:102}
\zeta(t)(fg)=\zeta(t)(f)\zeta(t)(g),\qquad\forall f,g\in C^{\omega}(M).
\end{equation}
\end{theorem}
\begin{proof}
Let $N=\mathrm{dim}(M)$ and $(V,(x^1,x^2,\ldots,x^N))$ be a coordinate chart around
$x_0$. Without loss of generality, we can assume that $\mathbb{T}$ is
a compact interval containing $t_0$.
Let $U$ be a relatively compact set such that
$\mathrm{cl}(U)\subseteq V$, $K\subseteq U$ be a compact set. For every $k\in\N$, we define
$\phi_k: \mathbb{T}\to \mathrm{L}(C^{\omega}(M);C^{\omega}(U))$ inductively as
\begin{eqnarray*}
\phi_0(t)(f)&=&f\mid_U,\qquad\forall t\in [t_0,T],\\
\phi_k(t)(f)&=&f\mid_{U}+\int_{t_0}^{t} \phi_{k-1}(\tau)\scirc \widehat{X}(\tau)(f)d\tau,\qquad\forall t\in \mathbb{T}.
\end{eqnarray*}
Let $K\subseteq M$ be a compact set and $\mathbf{a}\in \mathbf{c}^{\downarrow}_{0}(\Z_{\ge
  0},\R_{>0},6\textup{e}d)$. Then, we have the following lemma.
\begin{lemma*}
There exist a locally
integrally bounded function $m\in \mathrm{L}^1_{\loc}(\mathbb{T})$ such that, for every
$f\in C^{\omega}(M)$, there exist constants $M_f, \tilde{M}_f\in
\R^{\ge 0}$ 
\begin{eqnarray*}
p^{\omega}_{K,\mathbf{a},f}(\phi_n(t)-\phi_{n-1}(t))&\le&
                                                 (M(t))^nM_f,\qquad\forall
                                                 t\in \mathbb{T},\
                                                 \forall n\in \N.\\
p^{\omega}_{K,\mathbf{a},f}\left((\phi_n(t)-\phi_{n-1}(t))\scirc \widehat{X}(t)\right)&\le&
                                                 m(t)(M(t))^n\tilde{M}_f,\qquad\forall
                                                 t\in \mathbb{T},\ \forall n\in \N.
\end{eqnarray*}
where $M:\mathbb{T}\to \R$ is defined as 
\begin{equation*}
M(t)=\left|\int_{t_0}^{t}m(\tau)d\tau\right|,\qquad\forall t\in \mathbb{T}.
\end{equation*}
\end{lemma*}
\begin{proof}
Since $X$ is locally Bochner integrable, by Corollary \ref{th:13}, there exist a neighbourhood
$\overline{V}$ of $U$, a locally Bochner integrable vector field
$\overline{X}\in
\mathrm{L}^1(\mathbb{T};\Gamma_{\bdd}^{\hol,\R}(\overline{V}))$,
and a function $\overline{f}\in C_{\bdd}^{\hol,\R}(\overline{V})$
such that $\overline{X}_t$ and $\overline{f}$ are the holomorphic extension
of $X$ and $f$ over $\overline{V}$, respectively. Then, by Theorem
\ref{th:19}, there exists $d>0$ such that, for every compact set
$K\subseteq U$ and every $\mathbf{a}\in \mathbf{c}^{\downarrow}_{0}(\Z_{\ge
  0},\R_{>0},6\textup{e}d)$, we have
\begin{eqnarray*}
p^{\omega}_{K,\mathbf{a}}(f)&\le& p_{\overline{V}}(\overline{f}),\\
\max_i\left\{p^{\omega}_{K,\mathbf{a}}(X^i(t))\right\}&\le&
                                              \max_i\left\{p_{\overline{V}}(\overline{X}(t))\right\},\qquad\forall t\in \mathbb{T}, 
\end{eqnarray*}
Since $X$ is locally Bochner integrable, there exists $m\in
\mathrm{L}^1(\mathbb{T})$ such that
\begin{equation*}
4N\max_i\left\{p_{\overline{V}}(\overline{X}^i(t))\right\}\le m(t),\qquad\forall
t\in \mathbb{T},
\end{equation*}
Then we define $M:\mathbb{T}\to \R$ as
\begin{equation*}
M(t)=\int_{t_0}^{t} m(\tau)d\tau. 
\end{equation*}
Let $K\subseteq U$ be a compact set and let $\mathbf{a}\in
\mathbf{c}^{\downarrow}_{0}(\Z_{\ge 0},\R_{>0},d)$. We show by
induction that, for every $n\in N$, the function $\phi_{n}\scirc X$
is locally Bochner integrable and $\phi_{n+1}\in
\mathrm{AC}(\mathbb{T},\mathrm{L}(C^{\omega}(M);C^{\omega}(U)))$. Moreover, we have
\begin{equation*}
p^{\omega}_{K,\mathbf{a},f}(\phi_{n+1}(t)-\phi_n(t))\le
(M(t))^{n+1}p^{\omega}_{K,\mathbf{a}_{n+1}}(f),\qquad\forall
t\in \mathbb{T}, 
\end{equation*}
where, for every $n\in \N$, the sequence
$\mathbf{a}_n\in \mathbf{c}^{\downarrow}_{0}(\Z_{\ge 0},\R_{>0})$ is defined as in
Lemma \ref{lem:4}:
\begin{equation*}
a_{n,m}=
\begin{cases}
\left(\frac{m+1}{m}\right)^na_m& n<m,\\
\left(\frac{m+1}{m}\right)^ma_m& n\ge m.
\end{cases}
\end{equation*}
First note that for $n=0$, we have
\begin{equation*}
\phi_0\scirc \widehat{X}(f)=\widehat{X}(f)\mid_U,\qquad\forall f\in C^{\omega}(M),
\end{equation*}
Since $X$ is locally Bochner integrable, $\phi_0\scirc X$ is locally Bochner
integrable. Therefore, $\phi_1\in
\mathrm{AC}([t_0,T],\mathrm{L}(C^{\omega}(M);C^{\omega}(U)))$. Moreover,
we have
\begin{equation*}
\phi_{1}(t)-\phi_0(t)=\int_{t_0}^{t} \widehat{X}(\tau)d\tau,\qquad\forall t\in \mathbb{T}.
\end{equation*}
This implies that 
\begin{equation*}
p^{\omega}_{K,\mathbf{a},f}(\phi_{1}(t)-\phi_0(t))\le
\int_{t_0}^{t}
p^{\omega}_{K,\mathbf{a}}(\widehat{X}(\tau)f)d\tau,\qquad\forall t\in \mathbb{T}.
\end{equation*}
By inequality \eqref{eq:6}, we have 
\begin{equation*}
p^{\omega}_{K,\mathbf{a}}(X(t)f)\le 4N \max_i\{p^{\omega}_{K,\mathbf{b}_1}(X^i(t))\}p^{\omega}_{K,\mathbf{a}_1}(f),\qquad\forall
t\in \mathbb{T}.
\end{equation*}
Therefore we have
\begin{multline*}
p^{\omega}_{K,\mathbf{a},f}(\phi_{1}(t)-\phi_0(t))\le \int_{t_0}^{t}
4N\max_i\{p^{\omega}_{K,\mathbf{b}_1}(X^i(\tau))\}p^{\omega}_{K,\mathbf{a}_1}(f)d\tau
\\\le M(t)p^{\omega}_{K,\mathbf{a}_1}(f).
\end{multline*}
Now suppose that, for every $k\in \{1,2,\ldots,n-1\}$, $\phi_k\scirc X$
is locally Bochner integrable and we have
\begin{equation*}
p^{\omega}_{K,\mathbf{a},f}(\phi_{k+1}(t)-\phi_k(t))\le
(M(t))^{k+1}p^{\omega}_{K,\mathbf{a}_{k+1}}(f),\qquad\forall t\in\mathbb{T}.
\end{equation*}
Now consider the following inequality:
\begin{multline*}
p^{\omega}_{K,\mathbf{a},f}(\phi_{n-1}(t)\scirc \widehat{X}(t))\le p^{\omega}_{K,\mathbf{a},f}(\widehat{X}(t))+\sum_{i=1}^{n-1}
p^{\omega}_{K,\mathbf{a},f}((\phi_i(t)-\phi_{i-1}(t))\scirc \widehat{X}(t))\\\le
p^{\omega}_{K,\mathbf{a},f}(\widehat{X}(t))+\sum_{i=1}^{n-1} m(t)(M(t))^{i+1}\tilde{M}_f\le
m(t)\left(\sum_{i=0}^{n-1}
  (M(t))^i\right)\tilde{M}_f,\quad\forall t\in \mathbb{T}.
\end{multline*}
The function $g_n:[t_0,T]\to \R$ defined as 
\begin{equation*}
g_n(t)=m(t)\left(\sum_{i=0}^{n-1}M^i(t)\right),\qquad\forall t\in \mathbb{T},
\end{equation*}
is locally integrable. Thus, by Theorem \ref{th:3}, $\phi_{n-1}\scirc \widehat{X}$ is locally Bochner
integrable. So, by Definition \ref{def:2}, $\phi_n$ is absolutely continuous.

On the other hand, we have
\begin{equation*}
\phi_{n+1}(t)-\phi_n(t)=\int_{t_0}^{t} \left(\phi_{n}(\tau)\scirc
\widehat{X}(\tau)-\phi_{n-1}(\tau)\scirc \widehat{X}(\tau)\right)d\tau,\qquad\forall t\in \mathbb{T}.
\end{equation*}
Taking $p^{\omega}_{K,\mathbf{a},f}$ of both side of the above
equality, we have
\begin{multline*}
p^{\omega}_{K,\mathbf{a},f}(\phi_{n+1}(t)-\phi_n(t))\\\le \int_{t_0}^{t} p^{\omega}_{K,\mathbf{a},f}\left((\phi_{n}(\tau)-\phi_{n-1}(\tau))\scirc \widehat{X}(\tau) \right) d\tau,\qquad\forall t\in \mathbb{T}.
\end{multline*}
However, we know that by the induction hypothesis
\begin{equation*}
p^{\omega}_{K,\mathbf{a},f}\left((\phi_{n}(t)-\phi_{n-1}(t))\scirc
  \widehat{X}(t) \right)\le (M(t))^np^{\omega}_{K,\mathbf{a}_n}(\widehat{X}(t)f) ,\qquad\forall t\in \mathbb{T}.
\end{equation*}
Moreover, by the inequality \eqref{eq:6}, we have 
\begin{equation*}
p^{\omega}_{K,\mathbf{a}_{n}}(\widehat{X}(t)f)\le
4N(n+1)\max_i\left\{p^{\omega}_{K,\mathbf{b}_n}(X^i(t))\right\}
p^{\omega}_{K,\mathbf{a}_{n+1}}(f),\qquad\forall t\in \mathbb{T}.
\end{equation*}
By Lemma \ref{lem:4}, for every $n\in \N$, we have $\mathbf{b}_n\in
\mathbf{c}^{\downarrow}_{0}(\Z_{\ge 0},\R_{>0},6\textup{e}d)$. This implies that, for every
$n\in \N$, we
have
\begin{equation*}
\max_i\left\{p^{\omega}_{K,\mathbf{b}_n}(X^i(t))\right\}\le
\max_i\left\{p_{\overline{V}}(\overline{X}^i(t))\right\}<\frac{1}{4N}m(t),\quad\forall
t\in \mathbb{T}.
\end{equation*}
Therefore, for every $n\in \N$, we have
\begin{equation*}
p^{\omega}_{K,\mathbf{a},f}\left((\phi_{n}(t)-\phi_{n-1}(t))\scirc
  \widehat{X}(t)\right)\le (n+1)m(t)M^n(t)p^{\omega}_{K,\mathbf{a}_{n+1}}(f).
\end{equation*}
Thus we get
\begin{multline*}
p^{\omega}_{K,\mathbf{a},f}(\phi_{n+1}(t)-\phi_n(t))\\\le
\int_{t_0}^{t} (n+1)(M(\tau))^{n}m(\tau)
p^{\omega}_{K,\mathbf{a}_{n+1}}(f) d\tau\\=
(M(t))^{n+1}p^{\omega}_{K,\mathbf{a}_{n+1}}(f), \quad\forall t\in \mathbb{T}.
\end{multline*}
This completes the induction.
Note that by Lemma \ref{lem:4}, for every $m,n\in \Z_{\ge 0}$, we have
\begin{equation*}
a_{n,m}\le \textup{e}a_m\le 6\textup{e}d
\end{equation*}
This implies that, for every $n\in \N$, we have
\begin{equation*}
p^{\omega}_{K,\mathbf{a}_n}(f)\le p_{\overline{V}}(\overline{f}).
\end{equation*}
If we set $M_f=p_{\overline{V}}(\overline{f})$ then, for every $n\in \N$, we have
\begin{equation*}
p^{\omega}_{K,\mathbf{a},f}(\phi_{n+1}(t)-\phi_n(t))\le (M(t))^{n+1}M_f, \quad\forall t\in\mathbb{T}.
\end{equation*}
Moreover, for every $n\in\N$, we have
\begin{equation*}
p^{\omega}_{K,\mathbf{a},f}\left((\phi_{n}(t)-\phi_{n-1}(t))\scirc
  \widehat{X}(t)\right)\le (M(t))^np^{\omega}_{K,\mathbf{a}_n}(\widehat{X}(t)f), \quad\forall t\in\mathbb{T}.
\end{equation*}
However, by inequality \eqref{eq:6}, we have
\begin{equation*}
p^{\omega}_{K,\mathbf{a}^n}(\widehat{X}(t)f)\le
4N\max_{i}\left\{p^{\omega}_{K,\mathbf{b}_n}\right\}p^{\omega}_{K,\mathbf{a}_{n+1}}(f),\qquad\forall
t\in \mathbb{T}.
\end{equation*}
Noting that we have 
\begin{equation*}
\max_i\left\{p^{\omega}_{K,\mathbf{b}_n}(X^i(t))\right\}\le
\max_i\left\{p_{\overline{V}}(\overline{X}^i(t))\right\}<\frac{1}{4N}m(t),\quad\forall
t\in \mathbb{T},
\end{equation*}
and
\begin{equation*}
p^{\omega}_{K,\mathbf{a}_{n+1}}(f)\le p_{\overline{V}}(\overline{f}),\quad\forall
t\in \mathbb{T}.
\end{equation*}
Therefore, if we set
$\tilde{M}_f=p_{\overline{V}}(\overline{f})$, we have
\begin{equation*}
p^{\omega}_{K,\mathbf{a},f}\left((\phi_{n}(t)-\phi_{n-1}(t))\scirc
  \widehat{X}(t)\right)\le m(t)(M(t))^n\tilde{M}_f,\qquad\forall
t\in \mathbb{T}.
\end{equation*}
This completes the proof of the lemma.
\end{proof}

Therefore, for every $n\in \N$, we have 
\begin{equation*}
p^{\omega}_{K,\mathbf{a},f}(\phi_n(t)-\phi_{n-1}(t))\le (M(T))^nM_f,\qquad\forall t\in [t_0,T].
\end{equation*}
Since $M$ is continuous, there exists $T\in\mathbb{T}$ such that 
\begin{equation*}
M(t)<1,\qquad\forall t\in [t_0,T].
\end{equation*}
Since $M(T)<1$, one can deduce that the sequence $\{\phi_{n}\}_{n\in \N}$ converges
uniformly on $[t_0,T]$ in $\LC{M}{U}{\omega}$. Since uniform convergence
implies $L^1$-convergence and the space $\mathrm{L}^1([t_0,T];\LC{M}{U}{\omega})$
is complete, there exists $\phi\in \mathrm{L}^1([t_0,T];\LC{M}{U}{\omega})$ such
that 
\begin{equation*}
\lim_{n\to\infty} \phi_n=\phi,
\end{equation*}
where the limit is in $\mathrm{L}^1$-topology on
$\mathrm{L}^1([t_0,T];\LC{M}{U}{\omega})$. We first show that
$\phi\scirc X$ is locally Bochner integrable on $[t_0,T]$.
Note that, by the above Lemma, for every $n\in \N$, we have
\begin{equation}\label{eq:9}
p^{\omega}_{K,\mathbf{a},f}\left(\phi(t)-\phi_n(t)\right)\le
\sum_{k=n+1}^{\infty} (M(t))^kM_f.
\end{equation}
This implies that, for every $n\in \N$,
\begin{multline*}
\int_{t_0}^{t}p^{\omega}_{K,\mathbf{a},f}\left((\phi(\tau)-\phi_n(\tau))\scirc
  \widehat{X}(\tau)\right)d\tau \le \int_{t_0}^{t}\sum_{k=n+1}^{\infty}
m(\tau)(M(\tau))^k\tilde{M}_f \\\le N(T-t_0)\sum_{i=n+1}^{\infty} (M(T))^n\tilde{M}_f,\qquad\forall t\in [t_0,T].
\end{multline*}
Therefore, we get 
\begin{multline*}
\int_{t_0}^{t}p^{\omega}_{K,\mathbf{a},f}\left(\phi(\tau)\scirc
  \widehat{X}(\tau)\right)d\tau \\\le \int_{t_0}^{t}p^{\omega}_{K,\mathbf{a},f}\left(\phi_n(\tau)\scirc
  \widehat{X}(\tau)\right)d\tau + \frac{\tilde{M}_f N(T-t_0)(M(T))^{n+1}}{1-M(T)},\qquad\forall t\in [t_0,T].
\end{multline*}
However, from the proof of the above Lemma, we know that
\begin{equation*}
\int_{t_0}^{t}p^{\omega}_{K,\mathbf{a},f}\left(\phi_n(\tau)\scirc
  \widehat{X}(\tau)\right)d\tau\le g_n(t)\tilde{M}_f, \qquad\forall n\in \N,\
\forall t\in [t_0,T],
\end{equation*}
where $g_n:[t_0,T]\to \R$ is locally integrable. Therefore, we define
the function $h_n:[t_0,T]\to \R$ as
\begin{equation*}
h_n(t)=g_n(t)\tilde{M}_f+\frac{\tilde{M}_f
  N(T-t_0)(M(T))^{n+1}}{1-M(T)},\qquad\forall t\in [t_0,T].
\end{equation*}
It is clear that $h_n$ is locally integrable and 
\begin{equation*}
\int_{t_0}^{t}p^{\omega}_{K,\mathbf{a},f}\left(\phi(\tau)\scirc
  \widehat{X}(\tau)\right)d\tau\le h_n(t).
\end{equation*}
This implies that $\phi\scirc \widehat{X}$ is locally Bochner
integrable. Moreover, using equation \eqref{eq:9}, we get 
\begin{equation*}
\lim_{n\to\infty} \int_{t_0}^{t} \phi_n(\tau)\scirc
\widehat{X}(\tau) d\tau=\int_{t_0}^{t}\phi(\tau)\scirc \widehat{X}(\tau)
d\tau,\qquad\forall t\in [t_0,T].
\end{equation*}
Therefore, we have
\begin{equation*}
\phi(t)=\lim_{n\to\infty}\phi_n(t)=\lim_{n\to\infty} \int_{t_0}^{t}
\phi_{n-1}(\tau)\scirc \widehat{X}(\tau)d\tau=\int_{t_0}^{t}\phi(\tau)\scirc
\widehat{X}(\tau) d\tau.
\end{equation*}
This shows that $\phi$ satisfies the initial value problem \eqref{eq:2}.

One can also show that
the sequence $\{\phi_n\}_{n\in\N}$ converges to $\phi$ in
$\mathrm{AC}([t_0,T];\LC{M}{U}{\omega})$. In order to show this, it suffices to show that,
for every compact set $K\subseteq U$ and every $f\in C^{\omega}(M)$, we have
\begin{equation*}
\lim_{n\to\infty} \int_{t_0}^{t}p^{\omega}_{K,\mathbf{a},f}\left(\frac{d\phi_{n+1}}{dt}-\frac{d\phi_{n}}{dt}\right)=0,\quad\forall
t\in [t_0,T].
\end{equation*}
Note that, for every $n\in \N$, we have 
\begin{equation*}
\frac{d\phi_{n+1}}{dt}=\phi_n(t)\scirc \widehat{X}(t),\qquad\mbox{ a.e., } t\in [t_0,T].
\end{equation*}
Therefore, it suffices to show that
\begin{equation*}
\lim_{n\to\infty} \int_{t_0}^{t}p^{\omega}_{K,\mathbf{a},f}(\phi_n(t)\scirc
\widehat{X}(t)-\phi_{n-1}(t)\scirc \widehat{X}(t))=0,\qquad\forall t\in [t_0,T].
\end{equation*}
But we know that, for every $n\in \N$, we have
\begin{multline*}
p^{\omega}_{K,\mathbf{a},f}(\phi_n(t)\scirc \widehat{X}(t)-\phi_{n-1}(t)\scirc \widehat{X}(t))\le\\
m(t)(M(t))^n\tilde{M}_f\le m(t)(M(t))^n\tilde{M}_f,\qquad \forall t\in [t_0,T].
\end{multline*}
So we have
\begin{multline*}
\int_{t_0}^{t}p^{\omega}_{K,\mathbf{a},f}(\phi_n(t)\scirc
\widehat{X}(t)-\phi_{n-1}(t)\scirc \widehat{X}(t))\le \frac{d}{(n+1)N}(M(T))^{n+1}\tilde{M}_f\\\le \frac{d}{(n+1)N}(M(T))^{n+1}\tilde{M}_f.
\end{multline*}
This complete the proof of convergence of $\{\phi_n\}_{n\in\N}$ in $\mathrm{AC}([t_0,T];\LC{M}{U}{\omega})$.
\end{proof}

Using Theorem \ref{th:1} and the multiplicative property \eqref{eq:102} of the solution of
the initial value problem \eqref{eq:2}, one
can show that the solution $\phi$ constructed in Theorem \ref{th:37}
is the flow of the time-varying real analytic vector field $X$.  

\begin{corollary}
Let $X:\mathbb{T}\times M\to TM$ be a locally integrally bounded real
analytic vector field. Let $t_0\in \mathbb{T}$, $x_0\in M$, and $\phi^X:\mathbb{T}'\times U\to M$ be the flow of
$X$ defined on a time interval $\mathbb{T}'\subseteq\mathbb{T}$ containing
$t_0$ and a state neighbourhood $U\subseteq M$ containing $x_0$. We
know that $\phi^X$ satisfies
the following initial value problem for every $x\in U$.
\begin{eqnarray}\label{eq:101}
\begin{split}
&\dot{\phi}^X(t,x)=X(t,\phi^X(t,x)),\quad \mbox{
  a.e. }t\in \mathbb{T}',\\
&\phi^X(t_0,x)=x.
\end{split}
\end{eqnarray}
Then there exists a positive real number $T\in \mathbb{T}'$ such that
$T>t_0$ and a neighbourhood $V$ of $x_0$ such that, for every $t\in
[t_0,T]$ and every $x\in V$, we have
\begin{multline*}
f(\phi^X(t,x))=f(x)\\+\sum_{i=1}^{\infty}
\int_{t_0}^{t}\int_{t_0}^{t_1}\ldots \int_{t_0}^{t_{i-1}} \widehat{X}(t_i)\scirc
\widehat{X}(t_{i-1})\scirc\ldots\scirc \widehat{X}(t_1)(f)(x)dt_idt_{i-1}\ldots dt_1.
\end{multline*}
\end{corollary}
\begin{proof}
By Theorem \ref{th:37}, there exist $T>0$, a neighbourhood $V\subseteq
U$ of $x_0$, and a locally absolutely continuous curve $\xi:[t_0,T]\to
\mathrm{L}(C^{\omega}(M);C^{\omega}(U))$ such that
\begin{equation}\label{eq:100}
\xi(t)(fg)=\xi(t)(f)\xi(t)(g),\quad\ \mbox{ a.e. }t\in [t_0,T].
\end{equation}
and, for every $t\in [t_0,T]$ and every $x\in V$, we have
\begin{multline*}
\xi(t)(f)(x)=f(x)\\+\sum_{i=1}^{\infty} \int_{t_0}^{t}\int_{t_0}^{t_1}\ldots \int_{t_0}^{t_{i-1}} \widehat{X}(t_i)\scirc
\widehat{X}(t_{i-1})\scirc\ldots\scirc \widehat{X}(t_1)(f)(x)dt_idt_{i-1}\ldots dt_1.
\end{multline*}
Since $\xi$ satisfies equation \eqref{eq:100}, by Theorem \ref{th:1}, there exists a map $\phi:[t_0,T]\times V\to M$ such
that
\begin{equation*}
\widehat{\phi}(t)=\xi(t),\qquad\ \mbox{ a.e. }t\in [t_0,T].
\end{equation*}
This implies that, for almost every $t\in [t_0,T]$ and every $x\in V$, we have 
\begin{multline*}
f(\phi(t,x))=\xi(t)(f)(x)\\=f(x)+\sum_{i=1}^{\infty} \int_{t_0}^{t}\int_{t_0}^{t_1}\ldots \int_{t_0}^{t_{i-1}} \widehat{X}(t_i)\scirc
\widehat{X}(t_{i-1})\scirc\ldots\scirc \widehat{X}(t_1)(f)(x)dt_idt_{i-1}\ldots dt_1.
\end{multline*}
Therefore, by the uniqueness of the solution of the differential
equation \ref{eq:101},
it suffice to show that $\phi$ satisfies differential equations
\eqref{eq:101}. Note that, for every $t\in [t_0,T]$, we have
\begin{equation*}
\frac{d\widehat{\phi}(t)}{dt}=\lim_{h\to\infty}\frac{\widehat{\phi}(t+h)-\widehat{\phi}(t)}{h}.
\end{equation*}
By applying $f\in C^{\omega}(M)$ and noting that the topology on
$\mathrm{L}(C^{\omega}(M);C^{\omega}(V))$ is topology of pointwise
convergence, for almost every $t\in [t_0,T]$ and every $x\in V$, we have
\begin{multline*}
\frac{d\widehat{\phi}(t)}{dt}(f)(x)=\lim_{h\to\infty}\frac{\widehat{\phi}(t+h)(f)-\widehat{\phi}(t)(f)}{h}(x)\\=\lim_{h\to\infty}\frac{f(\phi(t+h,x))-f(\phi(t,x))}{h}=\frac{d(f(\phi(t,x)))}{dt},\qquad\forall
f\in C^{\omega}(M).
\end{multline*}
On the other hand, for almost every $t\in [t_0,T]$ and every $x\in V$, we have
\begin{equation*}
\frac{d\widehat{\phi}(t)}{dt}(f)(x)=\widehat{\phi}(t)\scirc
\widehat{X}(t)(f)(x)=X(t,\phi(t,x))(f),\qquad\forall f\in C^{\omega}(M).
\end{equation*}
Therefore, we have
\begin{equation*}
\frac{d(f(\phi(t,x)))}{dt}=X(t,\phi(t,x))(f),\qquad\forall f\in
C^{\omega}(M),\ \mbox{ a.e. }t\in [t_0,T], \ \forall x\in V.
\end{equation*}
This implies that
\begin{equation*}
\dot{\phi}(t,x)=X(t,\phi(t,x)),\ \mbox{ a.e. }t\in [t_0,T], \ \forall x\in V.
\end{equation*}
\end{proof}

\section{The exponential map}\label{sec:11}

In this section, we study the relationship between locally
integrally bounded time-varying real analytic vector fields and their
flows. In order to define such a map connecting time-varying vector
fields and their flows, one should note that there may not exist a fixed interval
$\mathbb{T}\subseteq\R$ containing $t_0$ and a fixed open
neighbourhood $U\subseteq M$ of $x_0$, such that the flow of ``every''
locally integrally bounded time-varying vector field $X\in
\mathrm{L}^1(\R, \Gamma^{\omega}(TM))$ is defined on time
interval $\mathbb{T}$ and on neighbourhood $U$. The following example
shows this for a family of real analytic vector fields.
\begin{example}\label{ex:3}
Consider the family of vector fields $\{X_n\}_{n\in \N}$, where
$X_n:\R\times \R\to T\R\simeq \R^2$ is defined as 
\begin{equation*}
X_n(t,x)=(x,nx^2),\qquad\forall t\in \mathbb{T}, \ \forall x\in \R.
\end{equation*}
Let $\mathbb{T}=[-1,1]$. Then, for every $n\in \N$, the flow of $X_n$
is defined as
\begin{equation*}
\phi^{X_n}(t,x)=\frac{x}{1-nxt}.
\end{equation*}
This implies that $\phi^{X_n}$ is only defined for $x\in
[-\frac{1}{n}, \frac{1}{n}]$. Therefore, there does not exist an open
neighbourhood $U$ of $0$ such that, for every $n\in \N$, $\phi^{X_n}$ is
defined on $U$.
\end{example}
The above example suggest that it is natural to define the connection
between vector fields and their flows on their germs around $t_0$ and $x_0$. Let $\mathbb{T}\subseteq\R$ be a compact interval containing $t_0\in \R$ and $U\subseteq M$
be an open set containing $x_0\in M$. We define 
\begin{equation*}
\LL{\omega}=\varinjlim \mathrm{L}^1(\mathbb{T}; \Gamma^{\omega}(TM)),
\end{equation*}
and
\begin{equation*}
\AC{\omega}=\varinjlim \ACC{T}{M}{U}{\omega}.
\end{equation*}
These direct limits are in the category of topological spaces. We define the exponential map $\exp: \LL{\omega} \to \AC{\omega}$ as
\begin{equation*}
\exp([X]_{(t_0,x_0)})=[\phi^X]_{(t_0,x_0)},\qquad\forall [X]_{(t_0,x_0)}\in \LL{\omega}.
\end{equation*}
\begin{theorem}\label{th:8}
The exponential map is sequentially continuous.
\end{theorem}
\begin{proof}

To show that $\exp:\LL{\omega} \to \AC{\omega}$ is a sequentially continuous map, it suffices to prove that, for
every sequence $\{X_n\}_{n\in\N}$ in
$\mathrm{L}^1(\mathbb{T};\Gamma^{\omega}(TM))$ which converges to
$X\in \mathrm{L}^1(\mathbb{T};\Gamma^{\omega}(TM))$,
the sequence $\{[\phi^{X_n}]_{(t_0,x_0)}\}$ converges to
$[\phi^X]_{(t_0,x_0)}$ in $\mathrm{AC}^{\omega}_{(t_0,x_0)}$. Since the sequence $\{X_n\}_{n\in \N}$ is converging, it is bounded in
$\mathrm{L}^1(\mathbb{T};\Gamma^{\omega}(TM))$. So, by Theorem
\ref{th:37}, there exists $T>t_0$ and a relatively compact coordinate neighbourhood $U$ of $x_0$ such that $[t_0,T]\subseteq
\mathbb{T}$ and, for every $n\in\N$, we have $\phi^{X_n}\in
\mathrm{AC}([t_0,T]; \LC{M}{U}{\omega})$. Therefore, it suffices to show that, for
the sequence $\{X_n\}_{n\in\N}$ in
$\mathrm{L}^1(\mathbb{T};\Gamma^{\omega}(TM))$ converging to $X\in
\mathrm{L}^1(\mathbb{T};\Gamma^{\omega}(TM))$,
the sequence $\{\phi^{X_n}\}$ converges to
$\phi^X$ in $\mathrm{AC}([t_0,T]; \LC{M}{U}{\omega})$. 

Let $f\in C^{\omega}(M)$ be a real analytic function and suppose that we have
\begin{equation*} 
\lim_{m\to\infty} X_m=X
\end{equation*}
in $\mathrm{L}^1(\mathbb{T};\Gamma^{\omega}(U))$. By Theorems \ref{th:21} and \ref{th:28}, there
exists a neighbourhood $\overline{V}\subseteq M^{\C}$ of $U$ such that
the bounded sequence of locally integrally bounded real analytic vector
fields $\{X_m\}_{m\in\N}$, the real analytic vector field $X$,
and the real analytic function $f$ can be extended to a converging
sequence of locally integrally bounded holomorphic vector fields
$\{\overline{X}_m\}_{m\in \N}$, a locally integrally bounded holomorphic vector field $X$, and a holomorphic function
$\overline{f}$ respectively. Moreover, by Theorem \ref{th:28}, the
inductive limit 
\begin{equation*}
\varinjlim
\mathrm{L}^1(\mathbb{T};\Gamma^{\hol,\R}_{\bdd}(\overline{U}_n))=\mathrm{L}^1(\mathbb{T};\Gamma^{\omega}(TM))
\end{equation*}
 is boundedly retractive. Therefore, we have
\begin{equation*}
\lim_{m\to\infty} \overline{X}_m=\overline{X}
\end{equation*}
in $\mathrm{L}^1(\mathbb{T}; \Gamma^{\hol,\R}_{\bdd}(\overline{V}))$.
Now, according to Theorem \ref{th:19}, there exists
$d>0$, such that for every compact set $K\subseteq U$, every
$\mathbf{a}\in \mathbf{c}^{\downarrow}_{0}(\Z_{\ge 0},\R_{>0}, d)$, and every $t\in \mathbb{T}$, we have 
\begin{eqnarray*}
p^{\omega}_{K,\mathbf{a}}(f)&\le&
                                  p_{\overline{V}}(\overline{f}),\\
\max_i\left\{p^{\omega}_{K,\mathbf{a}}(X^i(t))\right\}&\le&
                                  \max_i\left\{p_{\overline{V}}(\overline{X}^i(t))\right\},\\
\max_i\left\{p^{\omega}_{K,\mathbf{a}}(X^i(t)-X^i_m(t))\right\}&\le& \max_{i}\left\{p_{\overline{V}}(\overline{X}^i(t)-\overline{X}^i_m(t))\right\}.
\end{eqnarray*} 
Since $\overline{X}$ is locally integrally bounded, there exists $g\in
\mathrm{L}^1(\mathbb{T})$ such that 
\begin{equation*}
\max_i\left\{p_{\overline{V}}(\overline{X}^i(t))\right\}<g(t),\qquad\forall
t\in \mathbb{T}.
\end{equation*}
This implies that, for every compact set $K\subseteq U$ and every
$\mathbf{a}\in \mathbf{c}^{\downarrow}_{0}(\Z_{\ge 0},\R_{>0}, d)$, we have 
\begin{equation*}
\max_i\left\{p^{\omega}_{K,\mathbf{a}}(X^i(t))\right\}\le
\max_i\left\{p_{\overline{V}}(\overline{X}^i(t))\right\}<g(t),\quad\forall
t\in \mathbb{T}.
\end{equation*}
This means that, for every $\epsilon>0$, there exists $C\in \N$ such that 
\begin{equation*}
\int_{t_0}^{t}\max_i\left\{p_{\overline{V}}(\overline{X}^i_m(\tau)-\overline{X}^i(\tau))\right\}d\tau<\epsilon,\qquad\forall
m>C, \ t\in \mathbb{T}.
\end{equation*}
Therefore, if $m>C$, we have
\begin{equation*}
\max_i\left\{p_{\overline{V}}(\overline{X}^i_m(t))\right\}\le \max_i\left\{p_{\overline{V}}(\overline{X}^i(t))\right\}+
\epsilon\le g(t)+\epsilon,\qquad\forall t\in \mathbb{T},\ \forall m>C.
\end{equation*}
We define $m\in \mathrm{L}^1(\mathbb{T})$ as 
\begin{equation*}
m(t)=g(t)+\epsilon,\qquad\forall t\in\mathbb{T}.
\end{equation*}
We also define $\tilde{m}\in C(\mathbb{T})$ as 
\begin{equation*}
\tilde{m}(t)=\int_{t_0}^{t} (4N)m(\tau)d\tau,\qquad\forall
t\in \mathbb{T}.
\end{equation*}
We choose $T>t_0$ such that $|\tilde{m}(T)|<\frac{1}{2}$. 

\begin{lemma*}
Let $K\subseteq U$ be a compact set and
$\mathbf{a}\in \mathbf{c}^{\downarrow}_{0}(\Z_{\ge
  0},\R_{>0},d)$. Then, for every $n\in \N$, we have
\begin{multline*}
p^{\omega}_{K,\mathbf{a},f}(\phi^X_n(t)-\phi^{X_m}_n(t))\le
\left(\sum_{r=0}^{n-1}
(r+1)(\tilde{m}(t))^{r}p^{\omega}_{K,\mathbf{a}_{r+1}}(f)\right)\times\\\int_{t_0}^{t}
\max_i\left\{p_{\overline{V}}(\overline{X}^i(\tau)-\overline{X}^i_m(\tau))\right\}d\tau,\quad\forall
t\in [t_0,T], \ \forall m>C,
\end{multline*}
where $\mathbf{a}_{k}$ is as defined in Lemma \ref{lem:4}.
\end{lemma*}
\begin{proof}
We prove this lemma using induction on $n\in \N$. We first check the
case $n=1$. For $n=1$, using Theorem \ref{th:54}, we have
\begin{multline*}
p^{\omega}_{K,\mathbf{a},f}(\phi^X_1(t)-\phi^{X_m}_1(t))=p^{\omega}_{K,\mathbf{a},f}\left(\int_{t_0}^{t}
\widehat{X}(\tau)-\widehat{X}_m(\tau)d\tau\right)\\\le \int_{t_0}^{t}
p^{\omega}_{K,\mathbf{a},f}
\left(\widehat{X}(\tau)-\widehat{X}_m(\tau)\right)d\tau\\\le p^{\omega}_{K,\mathbf{a}_1}(f)\int_{t_0}^{t}
\max_i\left\{p_{\overline{V}}(\overline{X}^i(\tau)-\overline{X}^i_m(\tau))\right\}d\tau,\quad\forall
t\in [t_0,T], \ \forall m>C,
\end{multline*}

Now assume that, for $j\in \{1,2,\ldots,n\}$, we have 
\begin{multline*}
p^{\omega}_{K,\mathbf{a},f}(\phi^X_j(t)-\phi^{X_m}_j(t))\le
\sum_{r=0}^{j-1}\left(
(r+1)(\tilde{m}(t))^{r}p^{\omega}_{K,\mathbf{a}_{r+1}}(f)\right)\times\\\int_{t_0}^{t}
\max_i\left\{p_{\overline{V}}(\overline{X}^i(\tau)-\overline{X}^i_m(\tau))\right\}d\tau,\quad\forall
t\in [t_0,T], \ \forall m>C.
\end{multline*}
We want to show that
\begin{multline*}
p^{\omega}_{K,\mathbf{a},f}(\phi^X_{n+1}(t)-\phi^{X_m}_{n+1}(t))\le
\sum_{r=0}^{n}\left(
(r+1)(\tilde{m}(t))^{r}p^{\omega}_{K,\mathbf{a}_{r+1}}(f)\right)\times\\\int_{t_0}^{t}
\max_i\left\{p_{\overline{V}}(\overline{X}^i(\tau)-\overline{X}^i_m(\tau))\right\}d\tau,\quad\forall
t\in [t_0,T], \ \forall m>C.
\end{multline*}
Note that one can write
\begin{multline*}
\phi^X_{n+1}(t)-\phi^{X_m}_{n+1}(t)=
\int_{t_0}^{t}(\phi^X_{n}(\tau)\scirc
\widehat{X}(\tau)-\phi^{X_m}_{n}\scirc
\widehat{X}_m(\tau))d\tau\\=\int_{t_0}^{t}\left(\phi^X_n(\tau)-\phi^{X_m}_n(\tau)\right)\scirc
\widehat{X}(\tau)d\tau\\+\int_{t_0}^{t} \phi^{X_m}_n(\tau)\scirc\left(\widehat{X}(\tau)-\widehat{X}_m(\tau)\right)d\tau\quad\forall t\in [t_0,T], \ \forall m>C.
\end{multline*}
Therefore, for every compact set $K\subseteq U$ and every $\mathbf{a}\in \mathbf{c}^{\downarrow}_{0}(\Z_{\ge 0},\R_{>0}, d)$, we have
\begin{multline*}
p_{K,\mathbf{a},f} (\phi^X_n(t)-\phi^{X_m}_n(t))\le \int_{t_0}^{t}p^{\omega}_{K,\mathbf{a},f}\left(\left(\phi^X_n(\tau)-\phi^{X_m}_n(\tau)\right)\scirc
\widehat{X}(\tau)\right)d\tau\\+\int_{t_0}^{t}
p^{\omega}_{K,\mathbf{a},f}\left(\phi^{X_m}_n(\tau)\scirc\left(\widehat{X}(\tau)-\widehat{X}_m(\tau)\right)\right)d\tau,\quad\forall
t\in [t_0,T], \ \forall m>C.
\end{multline*}
Note that, for every $\widehat{X},\widehat{Y}\in \mathrm{L}^1([t_0,T];\Gamma^{\omega}(TM))$, we have
\begin{equation*}
p^{\omega}_{K,\mathbf{a},f}\left(\phi^{X}_n(t)\scirc \widehat{Y}(t)\right)=p^{\omega}_{K,\mathbf{a},f}(\widehat{Y}(t))+\sum_{r=1}^{n}
p^{\omega}_{K,\mathbf{a},f}\left((\phi^X_{r}(t)-\phi^X_{r-1}(t))\scirc
  \widehat{Y}(t)\right)
\end{equation*}
Since, for every $r\in \N$, we have
\begin{equation*}
p^{\omega}_{K,\mathbf{a},f}\left(\phi^X_{r}(t)-\phi^X_{r-1}(t)\right)\le
(\tilde{m}(t))^rp^{\omega}_{K,\mathbf{a}_r}(f),\quad\forall t\in [t_0,T]
\end{equation*}
for every $\widehat{X},\widehat{Y}\in \mathrm{L}^1([t_0,T];\Gamma^{\omega}(TM))$, we have
\begin{equation*}
p^{\omega}_{K,\mathbf{a},f}\left(\phi^{X}_n(t)\scirc \widehat{Y}(t)\right)\le \sum_{r=0}^{n}
(\tilde{m}(t))^rp^{\omega}_{K,\mathbf{a}_r,f}\left(\widehat{Y}(t)\right),\quad\forall
t\in [t_0,T].
\end{equation*}
This implies that, for every $t\in [t_0,T]$ and every $m>C$, we have
\begin{multline*}
p^{\omega}_{K,\mathbf{a},f}\left(\phi^{X_m}_n(t)\scirc\left(\widehat{X}(t)-\widehat{X}_m(t)\right)\right)\le \sum_{r=0}^{n}
(\tilde{m}(t))^rp^{\omega}_{K,\mathbf{a}_r,f}\left(\widehat{X}(t)-\widehat{X}_m(t)\right)\\\le \sum_{r=0}^{n}\left((r+1)
(\tilde{m}(t))^rp^{\omega}_{K,\mathbf{a}_{r+1}}(f)\right)\max_{i}\left\{p_{\overline{V}}\left(\overline{X}^i(t)-\overline{X}^i_m(t)\right)\right\}.
\end{multline*}
Therefore, for every $t\in [t_0,T]$ and every $m>C$, we get
\begin{multline*}
p^{\omega}_{K,\mathbf{a},f} (\phi^X_{n+1}(t)-\phi^{X_m}_{n+1}(t))\\\le
\int_{t_0}^{t} \sum_{r=0}^{n-1} \left((r+1)(r+2)(\tilde{m}(t))^{r}m(t) p^{\omega}_{K,\mathbf{a}_{r+2}}(f)\right)\int_{t_0}^{t}
\max_i\left\{p_{\overline{V}}(\overline{X}^i(\tau)-\overline{X}^i_m(\tau))\right\}d\tau\\+\int_{t_0}^{t}
\sum_{r=0}^{n}\left((r+1)(\tilde{m}(\tau))^rp^{\omega}_{K,\mathbf{a}_{r+1}}(f)\right)\max_i\left\{p_{\overline{V}}\left(\overline{X}^i(t)-\overline{X}^i_m(t)\right)\right\}d\tau.
\end{multline*}
Using integration by parts, we have 
\begin{multline*}
p^{\omega}_{K,\mathbf{a},f} (\phi^X_{n+1}(t)-\phi^{X_m}_{n+1}(t))\le
\sum_{r=0}^{n} (r+1)(\tilde{m}(t))^{r} p^{\omega}_{K,\mathbf{a}_{r+1}}(f)\times\\\int_{t_0}^{t}
p_{\overline{V}}(\overline{X}^i(\tau)-\overline{X}^i_m(\tau))d\tau,\quad\forall
t\in [t_0,T], \ \forall m>C.
\end{multline*}
This completes the proof of the lemma
\end{proof}

Thus, for every $n\in \N$, we have
\begin{multline*}
p^{\omega}_{K,\mathbf{a},f}(\phi^X_n(t)-\phi^{X_m}_n(t))\\\le
\sum_{r=0}^{n-1}
(r+1)(\tilde{m}(t))^{r}p^{\omega}_{K,\mathbf{a}_{r+1}}(f)\left(\int_{t_0}^{t}
p_{\overline{V}}(\overline{X}^i(\tau)-\overline{X}^i_m(\tau))d\tau\right),\quad\forall
t\in [t_0,T], \ \forall m>C.
\end{multline*}
Since, for every $t\in [t_0,T]$, we have 
\begin{equation*}
|\tilde{m}(t)|<\frac{1}{2},
\end{equation*}
the series 
\begin{equation*}
\sum_{r=0}^{\infty}(r+1)(\tilde{m}(t))^{r}p^{\omega}_{K,\mathbf{a}_{r+1}}(f)
\end{equation*}
converges to a function $h(t)$, for every $t\in [t_0,T]$. By Lebesgue's
monotone convergence theorem, $h$ is integrable. This implies
that, for every $n\in \N$ and every $\mathbf{a}\in \mathbf{c}^{\downarrow}_{0}(\Z_{\ge
0},\R_{>0},d)$, 
\begin{equation*}
p^{\omega}_{K,\mathbf{a},f}\left(\phi^X_n(t)-\phi^{X_m}_n(t)\right)\le
h(t)
\int_{t_0}^{t}p_{\overline{V}}(\overline{X}^i(\tau)-\overline{X}^i_m(\tau))d\tau, \qquad\forall
t\in [t_0,T], \ \forall m>C.
\end{equation*}
Therefore, by taking the limit as $n$ goes to infinity of the left
hand side of the inequality, we have
\begin{equation*}
p^{\omega}_{K,\mathbf{a},f}\left(\phi^X(t)-\phi^{X_m}(t)\right)\le
h(t)
\int_{t_0}^{t}p_{\overline{V}}(\overline{X}^i(\tau)-\overline{X}^i_m(\tau))d\tau, \qquad\forall
t\in [t_0,T], \ \forall m>C.
\end{equation*}
This completes the proof of sequential continuity of $\exp$. 
\end{proof}

\bibliographystyle{plain} 
\bibliography{Ref.bib}

\end{document}